\documentclass[10pt]{amsart}
\usepackage{xcolor}
\definecolor{winered}{rgb}{0.8,0,0}
\definecolor{deepblue}{rgb}{0,0,0.8}
\usepackage[colorlinks=true]{hyperref}
\hypersetup{linkcolor=winered, citecolor=deepblue}
\usepackage{amsmath}
\usepackage{amsthm}
\usepackage{amssymb}
\usepackage{mathrsfs}
\usepackage{tikz-cd}
\newtheorem{thm}{Theorem}[section]
\newtheorem{prop}[thm]{Proposition}

\newtheorem{lemma}[thm]{Lemma}
\newtheorem{lem}[thm]{Lemma}
\theoremstyle{definition}
\newtheorem{df}[thm]{Definition}
\newtheorem{rmk}[thm]{Remark}

\newtheorem{none}[thm]{}

\theoremstyle{remark}

\numberwithin{equation}{thm}

\newcommand{\bA}{\mathbf{A}}

\newcommand{\bG}{\mathbf{G}}
\newcommand{\bH}{\mathbf{H}}

\newcommand{\bZ}{\mathbf{Z}}

\newcommand{\cF}{\mathcal{F}}
\newcommand{\cG}{\mathcal{G}}

\newcommand{\rC}{\mathrm{C}}
\newcommand{\rD}{\mathrm{D}}

\newcommand{\et}{{\acute{e}t}}

\DeclareMathOperator{\Hom}{Hom}

\DeclareMathOperator{\uHom}{\underline{Hom}}
\DeclareMathOperator{\Spec}{Spec}

\newcommand{\id}{{\rm id}}

\newcommand{\DM}{{\mathrm{DM}}}
\newcommand{\eff}{{\mathrm{eff}}}

\newcommand{\DMeffk}{{\DM^\eff(k)}}
\newcommand{\DMeffket}{\mathrm{DM}_{\acute{e}t}^{\mathrm{eff}}(k)}

\newcommand{\Sh}{\mathrm{Sh}}

\newcommand{\Alb}{\mathrm{Alb}}

\newcommand{\LAlb}{\mathrm{LAlb}}
\newcommand{\RPic}{\mathrm{RPic}}
\newcommand{\NS}{\mathrm{NS}}
\newcommand{\Pic}{\mathrm{Pic}}

\usepackage[shortlabels]{enumitem}
\begin{document}
\title{Motivic interpretation of Albanese varieties of smooth varieties}
\author{Doosung Park}
\date{\today}
\address{Institut f\"ur Mathematik, Universit\"at Z\"urich, Winterthurerstr. 190, 8057 Z\"urich, Switzerland}
\email{doosung,park@math.uzh.ch}
\begin{abstract}
For every noetherian smooth and separated scheme over an algebraically closed field, we define its derived Albanese in Voevodsky's triangulated category of effective Nisnevich motives.
To justify our definition, we relate the derived Albanese with the Albanese scheme.
We also prove that the derived Albanese satisfies the Nisnevich descent property.
\end{abstract}
\maketitle
\section{Introduction}

In the book of Barbieri-Viale and Kahn \cite{MR3545132}, they defined the \emph{derived Albanese} of $X$ for every $X\in Sm/k$ when $k$ is a perfect field.
Here, $Sm/k$ denotes the category of noetherian smooth and separated schemes over $k$.
Their definition can be interpreted as
\[
\LAlb(X)
:=
\uHom_{\DMeffket}(a_{\et}\uHom_\DMeffk(M(X),\bZ(1)[2]),\bZ(1)[2]),
\]
where $\DMeffk$ (resp.\ $\DMeffket$) is Voevodsky's triangulated category of effective Nisnevich (resp.\ \'etale) motives \cite{MVW}, $a_{\et}\colon \DMeffk\to \DMeffket$ is the derived sheafification functor, $M(X)$ is the motive of $X$, and $\uHom_\DMeffk$ and $\uHom_{\DMeffket}$ are the internal homs in $\DMeffk$ and $\DMeffket$.
Unfortunately, $\bZ/p\simeq 0$ in $\DMeffket$ if $p$ is the exponential characteristic of $k$, i.e., $p$ is inverted in their definition.

What we do in this paper is to define $\LAlb(X)$ in $\DMeffk$ without inverting $p$ when $k$ is an algebraically closed field.
Our definition is as follows:
\[
\LAlb(X)
:=
\tau_{\geq 0} \uHom_\DMeffk(\uHom_\DMeffk(M(X),\bZ(1)[2]),\bZ(1)[2]),
\]
where $\tau_{\geq 0}$ is the truncation functor for the homotopy $t$-structure (Definition \ref{1.2}).

We describe our $\LAlb(X)$ as follows to justify our definition.

\begin{thm}[Theorem \ref{2.41}]
Suppose $X\in Sm/k$.
Then there exists a functorial distinguished triangle
\[
\NS^*(X)[1]
\to
\LAlb(X)
\to
\Alb(X)
\to
\NS^*(X)[2]
\]
in $\DMeffk$, where $\NS^*(X):=\underline{\rm Hom}_{\DMeffk}(\NS(X),\bZ(1)[1])$, and $\NS(X)$ denotes the N\'eron-Severi group of $X$.
\end{thm}
We note that the derived Albanese of Barbieri-Viale and Kahn also satisfies a similar result after inverting $p$, see \cite[Theorem 9.2.2]{MR3545132}.
We stick to the Nisnevich topology, and this makes our proof different from theirs.

We also show that $\LAlb$ satisfies the Nisnevich descent property for another justification, see Theorem \ref{2.44} for details.

\begin{none}
\label{0.5}
\emph{Outline of the proof.}
We define motives $M_{\geq 1}(X)$ and $M_1^*(X)$ fitting in a distinguished triangle
\[
M_1^*(X)\to \uHom_{\DMeffk}(M_{\geq 1}(X),\bZ(1)[2]) \to \NS(X)\to M_1^*(X)[1],
\]
which will play a central role throughout the proof.
When $X$ is proper, the structure of $\uHom_{\DMeffk}(M(X),\bZ(1)[2])$ is analyzed in Proposition \ref{1.25}.
When $X$ is not proper, we can find an $h$-hypercover $X_\bullet\to X$ such that $X_\bullet$ admits a suitable compactification by de Jong's alterations.
We need to relate the two motives
\[
\uHom_{\DMeffk}(M(X),\bZ(1)[2])
\text{ and }
\uHom_{\DMeffk}(M(X_\bullet),\bZ(1)[2]),
\]
which is done in Proposition \ref{1.6}.
This allows us to deduce that there is a distinguished triangle
\begin{equation}
\label{0.5.2}
A\to M_1^*(X)\to N[1]\to A[1]
\end{equation}
in $\DMeffk$ for some lattice $N$ and abelian variety $A$.
This is the main content of Section 5.

Our next goal is to show that there is an isomorphism
\begin{equation}
\label{0.5.1}
\tau_{\geq 0}(\uHom_\DMeffk(A^\vee,\bZ(1)[2])) \simeq A
\end{equation}
for every abelian variety $A$, where $A^\vee$ is the dual abelian variety of $A$.
Together with the results in Sections 4 and 6 about the computations of certain hom groups, we deduce \eqref{0.5.1} from what we have discovered in Section 5.
The functoriality of \eqref{0.5.1} ends Section 7.
In Section 8, we generalize \eqref{0.5.1} to semi-abelian varieties.
We also establish the functoriality.
As a consequence, we relate \eqref{0.5.2} with the Albanese variety of $X$, and we finish the proof.
\end{none}

\begin{none}\label{0.4}
{\it Notations and convention.}
\begin{enumerate}
\item[(1)] $k$ is an algebraically closed field.
\item[(2)] $Sch/k$ denotes the category of noetherian separated schemes over $k$.
\item[(3)] ${\rm Sh}^{tr}(k)$ denotes the category of Nisnevich sheaves with transfers on $Sm/k$.
\item[(4)] A \emph{lattice} is the constant sheaf in $\Sh^{tr}(k)$ associated with a finitely generated free abelian group.
\item[(5)] For any morphism $F\rightarrow G$ in an abelian category, let $[F\rightarrow G]$ denote the naturally associated complex, where $F$ sits in degree $0$.
  \end{enumerate}
\end{none}

\begin{none}
\emph{Acknowledgement.}
Part of this work is done while the author stayed at Centre for Advanced Study at the Norwegian Academy of Science and Letters.
We would like to thank this institution for providing a very comfortable working environment.
\end{none}
\section{Picard functors}
\label{pic}

In this section, we study the Picard functors of simplicial schemes in $Sm/k$.

\begin{df}
\label{1.12}
Suppose that $X_\bullet$ is a simplicial scheme in $Sm/k$.
The \emph{simplicial Picard group of $X_\bullet$} is defined to be the hypercohomology group
\[
\Pic(X_\bullet):=\bH_{\et}^1(X_\bullet,\bG_m).
\]
The \emph{simplicial Picard functor} $\Pic_{X_\bullet/k}$ is a presheaf of abelian groups on $Sm/k$ given by
\[
\Pic_{X_\bullet/k}(T):=\Pic(X_\bullet\times T)/{\rm Pic}(T)
\]
for $T\in Sm/k$.
It has a transfer structure since the functors $T\mapsto {\rm Pic}(X_\bullet\times T)$ and $T\mapsto {\rm Pic}(T)$ are presheaves with transfers on $Sm/k$, see \cite[Example 2.5]{MVW}.
When $X_\bullet$ is a constant simplicial scheme $X$, $\Pic_{X/k}$ is the restriction of the usual Picard functor to $Sm/k$.
\end{df}

\begin{none}
Suppose $X\in Sm/k$.
If $T$ is an integral scheme in $Sm/k$, consider the composite homomorphism
\[
{\rm Pic}(X\times T)\stackrel{i^*}\rightarrow {\rm Pic}(X)\rightarrow {\rm NS}(X),
\]
where $i\colon X\to X\times T$ is the pullback of a closed immersion $\{x\}\rightarrow T$ from a rational point $x$ of $X$.
The definition of $\NS(X)$ tells that the composite homomorphism is independent of the choice of $i$.
Moreover, any element of ${\rm Pic}(T)$ in ${\rm Pic}(X\times T)$ maps to $0$ in ${\rm NS}(X)$.
Hence we have an induced morphism
  \[{\rm Pic}_{X/k}\rightarrow {\rm NS}(X)\]
of presheaves with transfers, where we regard ${\rm NS}(X)$ as a constant sheaf with transfers.
We denote by ${\rm Pic}_{X/k}^0$ its kernel.
\end{none}
\begin{none}
\label{1.15}
For every $X\in Sm/k$, we set
\[
\pi_0(\bZ^{tr}(X)):=\bZ^{r}
\]
if $X$ has exactly $r$ connected components.
We denote by $\bZ_{\geq 1}^{tr}(X)$ the kernel of the induced morphism
\[
\bZ^{tr}(X)\rightarrow \pi_0(\bZ^{tr}(X)).
\]
Since $k$ is algebraically closed, every connected component of $X$ has a rational point.
Thus there is a noncanonical decomposition
\[
\bZ^{tr}(X)\simeq  \bZ_{\geq 1}^{tr}(X)\oplus \pi_0(\bZ^{tr}(X)).
\]
We set
\[
M_{\geq 1}(X):=C_*\bZ_{\geq 1}^{tr}(X)\text{ and }M_0(X):=\pi_0(\bZ^{tr}(X)),
\]
and we view them as objects of $\DMeffk$.
If $X_\bullet$ is a simplicial scheme in $Sm/k$, we define $\pi_0(\bZ^{tr}(X_\bullet))$ and $\bZ_{\geq 1}^{tr}(X_\bullet)$ similarly, and we set
\[
M_{\geq 1}(X_\bullet):=\mathrm{Tot}(C_*\bZ_{\geq 1}^{tr}(X_\bullet))\text{ and }
M_0(X_\bullet):= \pi_0(\bZ^{tr}(X_\bullet)).
\]
\end{none}

\begin{prop}
\label{1.13}
For every $X\in Sm/k$, ${\rm Pic}_{X/k}$ is a Nisnevich sheaf with transfers on $Sm/k$.
\end{prop}
\begin{proof}
Since we have checked that ${\rm Pic}_{X/k}$ is a presheaf with transfers in Definition \ref{1.12}, it remains to show that ${\rm Pic}_{X/k}$ is a Nisnevich sheaf.
Suppose $T\in Sm/k$.
If $T$ is a disjoint union of $T_1$ and $T_2$, then
\[
{\rm Pic}(X\times T)\simeq {\rm Pic}(X\times T_1)\oplus {\rm Pic}(X\times T_2)
\text{ and }
{\rm Pic}(T)\simeq{\rm Pic}(T_1)\oplus {\rm Pic}(T_2).\]
  Thus
  \begin{equation}\label{1.13.1}
  {\rm Pic}_{X/k}(T)\simeq {\rm Pic}_{X/k}(T_1)\oplus {\rm Pic}_{X/k}(T_2).
  \end{equation}
Suppose that
  \[\begin{tikzcd}
    U'\arrow[r,"g'"]\arrow[d,"f'"']&T'\arrow[d,"f"]\\
    U\arrow[r,"g"]&T
  \end{tikzcd}\]
is a Nisnevich distinguished square, i.e., $f$ is \'etale, $g$ is an open immersion, and the induced morphism $f^{-1}(T-g(U))\rightarrow T-g(U)$ is an isomorphism with the reduced scheme structure on $T-g(U)$.
To check that ${\rm Pic}_{X/k}$ is a Nisnevich sheaf, owing to \cite[Corollary 2.17]{Voe10a} it suffices to show that the induced sequence
\[
0\rightarrow {\rm Pic}_{X/k}(T)\stackrel{p''}\rightarrow {\rm Pic}_{X/k}(T')\oplus {\rm Pic}_{X/k}(U)\stackrel{q''}\rightarrow {\rm Pic}_{X/k}(U')\rightarrow 0
\]
is exact.
By \eqref{1.13.1}, we reduce to the case when $T$ is connected.
Then $T$ is integral since $T$ is smooth over $k$.

Let $d$ be the number of irreducible components of $T-g(U)$ whose dimensions are $({\rm dim}\,T-1)$.
By Bloch's localization sequence \cite{zbMATH00653321}, there are commutative diagrams with exact rows:
\[
\begin{tikzcd}
    \bZ^d\arrow[d,"{\rm id}"']\arrow[r,"u"]&{\rm Pic}(T)\arrow[d,"f^*"]\arrow[r,"g^*"]&{\rm Pic}(U)\arrow[d,"f'^*"]\arrow[r]&0\\
    \bZ^d\arrow[r]&{\rm Pic}(T')\arrow[r,"g'^*"]&{\rm Pic}(U')\arrow[r]&0,
\end{tikzcd}
\]
\[
\begin{tikzcd}
    \bZ^d\arrow[d,"{\rm id}"']\arrow[r,"v"]&{\rm Pic}(X\times T)\arrow[d]\arrow[r]&{\rm Pic}(X\times U)\arrow[d]\arrow[r]&0\\
    \bZ^d\arrow[r]&{\rm Pic}(X\times T')\arrow[r]&{\rm Pic}(X\times U')\arrow[r]&0.
  \end{tikzcd}
\]

There is an induced commutative diagram
  \begin{equation}\label{1.13.2}
    \begin{tikzcd}[column sep=small, row sep=small]
    {\rm Pic}(T)\arrow[r,"p"]\arrow[d,"r"']&{\rm Pic}(T')\oplus {\rm Pic}(U)\arrow[d,"r'"]\arrow[r,"q"]&{\rm Pic}(U')\arrow[d,"r''"]\arrow[r]\arrow[d]&0\\
    {\rm Pic}(X\times T)\arrow[r,"p'"]\arrow[d,"s"']&{\rm Pic}(X\times T')\oplus {\rm Pic}(X\times U)\arrow[r,"q'"]\arrow[d,"s'"]&{\rm Pic}(X\times U')\arrow[r]\arrow[d,"s''"]& 0\\
    {\rm Pic}_{X/k}(T)\arrow[r,"p''"]\arrow[d]& {\rm Pic}_{X/k}(T')\oplus {\rm Pic}_{X/k}(U)\arrow[r,"q''"]\arrow[d]& {\rm Pic}_{X/k}(U')\arrow[r]\arrow[d]& 0\\
    0&0&0,
  \end{tikzcd}
  \end{equation}
where $p=(f^*,-g^*)$ and $q$ is the summation of $f'^*$ and $g'^*$.
Note that all the columns are exact by definition.
Taking ${\rm Hom}_\DMeffk(-,\bZ(1)[2])$ to the distinguished triangle
\[
M(U')\rightarrow M(T')\oplus M(U)\rightarrow M(T)\rightarrow M(U')[1],
\]
we deduce that the top row in \eqref{1.13.2} is exact.
Similarly, the middle row in \eqref{1.13.2} is exact.

Let us show that the bottom row in \eqref{1.13.2} is exact. 
Consider an element
\[
b''\in {\rm Pic}_{X/k}(T')\oplus {\rm Pic}_{X/k}(U)
\]
such that $q''(b'')=0$.
Choose
\[
b'\in {\rm Pic}(X\times T')\oplus {\rm Pic}(X\times U)
\]
such that $s'(b')=b''$.
Then $s''(q'(b'))=q''(s'(b'))=0$, so $q'(b')=r''(c)$ for some $c\in {\rm Pic}(U')$.
Choose
\[
b\in {\rm Pic}(T')\oplus {\rm Pic}(U)
\]
such that $q(b)=c$.
Then $q'(b'-r'(b))=q'(b')-r''(c)=0$, so $b'-r'(b)=p'(a')$ for some
\[
a'\in {\rm Pic}(X\times T).
\]
It follows that
\[
b''=s'(b')=s'(b'-r'(b))=s'(p'(a))=p''(s(a')),
\]
so $b''$ is in the image of $p''$.
Since $q'$ and $s''$ are surjective, $q''$ is also surjective. We have shown that the bottom row in \eqref{1.13.2} is exact.

It remains to show that $p''$ is injective.
Let $a''\in {\rm Pic}_{X/k}(T)$ be an element such that $p''(a'')=0$.
Choose
\[
a'\in {\rm Pic}(X\times T)
\]
such that $s(a')=a''$.
Then $s'(p'(a'))=p''(a'')=0$, so $p'(a')=r'(b)$ for some
\[
b\in {\rm Pic}(T')\oplus {\rm Pic}(U).
\]
Thus $r''(q(b))=q'(r'(b))=q'(p'(a'))=0$.
Since $k$ is algebraically closed, the projection $X\times U'\rightarrow U'$ has a section. Thus $r''$ is injective, and hence $q(b)=0$ since $r''(q(b))=0$. 
It follows that $b=p(a)$ for some
\[
a\in {\rm Pic}(T).
\]
Now $p'(a'-r(a))=p'(a')-r'(b)=0$.
It follows that $v(t)=a'-r(a)$ for some $t\in \bZ^d$.
Then $a'-r(a)=r(u(t))$, so $a'$ is in the image of $r$.
Thus $a''=0$. This proves that $p''$ is injective.
\end{proof}

\begin{df}\label{3.26}
  Recall that an object $F$ of a triangulated category $\mathcal{T}$ is {\it compact} if the functor ${\rm Hom}_\mathcal{T}(F,-)$ commutes with small sums.
\end{df}
\begin{df}\label{2.14}
  Let $\mathcal{T}$ be a triangulated category having small sums, and let $\mathcal{F}$ be an essentially small class of compact objects in $\mathcal{T}$. Recall from \cite[Proposition 2.1.70]{Ayo07} that there is a $t$-structure such that the category of $t$-positive objects is the smallest full subcategory of $T$ containing $\mathcal{F}$ and stable under small sums, suspensions, and extensions. This $t$-structure is called the \emph{$t$-structure on $\mathcal{T}$ generated by $\mathcal{F}$}.

  For $i\in \bZ$, we denote by $h_i$ the homology functor, and we denote by $\tau_{\leq i}$ and $\tau_{\geq i}$ the homological truncation functors.
  
  According to the definition and properties of $t$-structures, we have the following.
  \begin{enumerate}
    \item[(i)] For every $M\in \mathcal{T}$, $\tau_{\geq 0}M$ is $t$-positive, and $\tau_{\leq 0}M$ is $t$-negative.
    \item[(ii)] For every $t$-positive object $M$ and $t$-negative object $N$, we have the vanishing
\[
{\rm Hom}_\mathcal{T}(M,N[-1])=0.
\]
    \item[(iii)] For every $M\in \mathcal{T}$ and integer $i$, there is a canonical distinguished triangle
    \[\tau_{\geq i}M\rightarrow M\rightarrow \tau_{\leq i-1}M\rightarrow \tau_{\geq i}M[1].\]
    \item[(iv)] For every $M\in \mathcal{T}$ and integer $i$,
    \[h_i(M):=(\tau_{\geq i}\tau_{\leq i}M)[-i]\simeq (\tau_{\leq i}\tau_{\geq i}M)[-i],\]
    and this is in the heart.
  \end{enumerate}
\end{df}
\begin{prop}\label{1.34}
For every $X\in Sm/k$, $M(X)$ and $M_{\geq 1}(X)$ are compact in $\DMeffk$.
\end{prop}
\begin{proof}
By \cite[Example 5.1.29(2), Proposition 5.1.32]{CD19}, we see that $M(X)$ is compact in $\DMeffk$.
In particular, $\bZ\simeq M(k)$ is compact in $\DMeffk$.
From the distinguished triangle
\[
M_{\geq 1}(X)\rightarrow M(X)\rightarrow \bZ^{\pi_0(X)}\rightarrow M_{\geq 1}(X)[1],
\]
we deduce that $M_{\geq 1}(X)$ is compact in $\DMeffk$ too.
\end{proof}
\begin{df}\label{1.2}
The {\it $0$-motivic $t$-structure} (or \emph{homotopy $t$-structure}) on $\DMeffk$ is the $t$-structure generated by objects of the form
\[M(X)\]
for all $X\in Sm/k$.
Due to \cite[Proposition 3.3]{MR2735752}, this definition is equivalent to the definition in \cite[Definition 3.1]{MR2735752}.
Note that the heart of the $0$-motivic $t$-structure is equivalent to the category of homotopy invariant Nisnevich sheaves with transfers on $Sm/k$ by the following paragraph of \cite[Definition 3.1]{MR2735752}.
\end{df}

\begin{prop}
\label{1.20}
For every $X\in Sm/k$, there is an isomorphism
\[
h_0(\underline{\rm Hom}_\DMeffk(M(X),\bZ(1)[2]))\simeq{\rm Pic}_{X/k}.
\]
\end{prop}
\begin{proof}
For every $T\in Sm/k$, there is an isomorphism
\[
\Hom_\DMeffk(M(T),\underline{\rm Hom}_\DMeffk(M(X),\bZ(1)[2]))
\simeq
\Pic(T\times X).
\]
Thus $h_0(\underline{\rm Hom}_\DMeffk(M(X),\bZ(1)[2]))$ is the Nisnevich sheaf with transfers associated with the presheaf with transfers
\[
T\mapsto {\rm Pic}(T\times X).
\]
By Proposition \ref{1.13}, ${\rm Pic}_{X/k}$ is a Nisnevich sheaf.
There is a morphism
\[
p\colon h_0(\underline{\rm Hom}_\DMeffk(M(X),\bZ(1)[2]))\rightarrow {\rm Pic}_{X/k}
\]
given by taking the quotient
\[
q\colon {\rm Pic}(T\times X)\rightarrow {\rm Pic}(T\times X)/{\rm Pic}(T).
\]
Let us show $p$ is an isomorphism. It suffices to check that $q$ is an isomorphism when $T$ is a henselian local scheme.
In this case, we are done since ${\rm Pic}(T)=0$.
\end{proof}

\begin{prop}
\label{1.26}
For every proper simplicial scheme $X_\bullet$ in $Sm/k$, there is an isomorphism
\[
h_0(\underline{\rm Hom}_{\DMeffk}(M(X_\bullet),\bZ(1)[2]))
\simeq
\Pic_{X_{\bullet}/k}.
\]
\end{prop}
\begin{proof}
We can argue as in the proof of Proposition \ref{1.20}.
Just observe that $\Pic_{X_\bullet/k}$ is a Nisnevich sheaf since it is representable, see \cite[Lemma 4.1.2]{MR1891270}.
\end{proof}

\begin{prop}
\label{1.25}
For every proper $X\in Sm/k$, there is an isomorphism
\[
\underline{\rm Hom}_\DMeffk(M_{\geq 1}(X),\bZ(1)[2])\simeq {\rm Pic}_{X/k}.
\]
\end{prop}
\begin{proof}
We only need to consider the case when $X$ is integral.
By Proposition \ref{1.20}, for every $i\in \bZ-\{0\}$ it suffices to show that the morphism
  \[h_i(\underline{\rm Hom}_\DMeffk(M(k),\bZ(1)[2]))\rightarrow h_i(\underline{\rm Hom}_\DMeffk(M(X),\bZ(1)[2]))\]
induced by the structure morphism $X\rightarrow k$ is an isomorphism.
These Nisnevich sheaves are associated with the presheaves
  \[T\mapsto H_{Nis}^{1-i}(X\times T,\bG_m)\text{ and } T\mapsto H_{Nis}^{1-i}(T,\bG_m).\]
According to \cite[Vanishing Theorem 19.3]{MVW}, these are $0$ if $i\neq 0,1$.
Thus we only need to show that the induced homomorphism
\begin{equation}\label{1.25.1}
H^0(T,\bG_m)\rightarrow H^0(X\times T,\bG_m)
\end{equation}
is an isomorphism for every $T\in Sm/k$.
We also only need to consider the case when $T$ is integral.

Let us argue as in \cite[Lemma 2.12]{Man}.
Since $k$ is algebraically closed, the projection $p\colon X\times T \rightarrow T$ has a section $i\colon T\rightarrow X\times T$.
Hence it suffices to show that \eqref{1.25.1} is surjective.
By \cite[Th\'eor\`eme III.7.7.6]{EGA}, the induced homomorphism
\[
p^*\colon H^0(T,\mathcal{O}_T)\rightarrow H^0(X\times T,\mathcal{O}_{X\times T})
\]
is an isomorphism.
Let $a\in H^0(X\times T,\bG_m)$ be an element. Then $a=p^*b$ for some $b\in H^0(T,\mathcal{O}_T)$. Since $i^*a=i^*p^*b=b$, $b$ is in $H^0(T,\bG_m)$. This shows that \eqref{1.25.1} is surjective.
\end{proof}

\begin{prop}
\label{1.39}
Suppose that $X_\bullet$ is a proper simplicial scheme in $Sm/k$.
If $L$ is the kernel of
\[
\bZ^{\pi_0(X_0)}\rightarrow \bZ^{\pi_0(X_1)}
\]
that is the dual of the induced morphism $\bZ^{\pi_0(X_1)}\rightarrow \bZ^{\pi_0(X_0)}$, then
\[
h_i(\uHom_{\DMeffk}(M(X_\bullet),\bZ(1)[2]))
\simeq
\left\{
\begin{array}{ll}
\Pic_{X_{\bullet}/k} & \text{if }i=0,
\\
L\otimes \bG_m & \text{if }i=1,
\\
0 & \text{if }i\geq 2.
\end{array}
\right.
\]
\end{prop}
\begin{proof}
By Proposition \ref{1.26}, we only need to work for the case when $i\geq 1$.
Let $\rC\Pic_{X_\bullet/k}$ denote the associated complex
\[
\Pic_{X_0/k}\rightarrow \Pic_{X_1/k}\rightarrow \cdots,
\]
where $\Pic_{X_0/k}$ sits in degree $0$.
Due to Proposition \ref{1.25}, there is a distinguished triangle
\begin{align*}
&\uHom_{\DMeffk}(M_0(X_\bullet),\bZ(1)[2])
\rightarrow
\underline{\rm Hom}_{\DMeffk}(M(X_\bullet),\bZ(1)[2])
\\
\rightarrow 
&\rC\Pic_{X_\bullet/k}
\rightarrow
\Hom_{\DMeffk}(M_0(X_\bullet),\bZ(1)[2])[1].
\end{align*}
This gives a long exact sequence
\begin{align*}
\cdots
\rightarrow
& h_i(\uHom_{\DMeffk}(M_0(X_\bullet),\bZ(1)[2]))
\rightarrow
h_i(\uHom_{\DMeffk}(M(X_\bullet),\bZ(1)[2]))
\\
\rightarrow
& h_i(\rC\Pic_{X_\bullet/k})
\rightarrow
h_{i-1}(\Hom_{\DMeffk}(M_0(X_\bullet),\bZ(1)[2])[1]))
\rightarrow
\cdots.
\end{align*}
Observe that
\begin{gather*}
h_i(\rC\Pic_{X_\bullet/k})=0\text{ for }i\geq 1,
\\
h_i(\uHom_{\DMeffk}(M_0(X_\bullet),\bZ(1)[2]))=0\text{ for }i\geq 2,
\\
h_1(\uHom_{\DMeffk}(M_0(X_\bullet),\bZ(1)[2]))\simeq L\otimes \bG_m.
\end{gather*}
Plug these in the above long exact sequence to finish the proof.
\end{proof}

\begin{df}
\label{1.31}
For every $X\in Sm/k$, there is a composite morphism
\[
\begin{split}
\underline{\rm Hom}_\DMeffk(M_{\geq 1}(X),\bZ(1)[2])\rightarrow &h_0(\underline{\rm Hom}_\DMeffk(M_{\geq 1}(X),\bZ(1)[2]))
\\
\stackrel{\sim}\rightarrow &{\rm Pic}_{X/k}\rightarrow {\rm NS}(X),
\end{split}
\]
where the second arrow is given in Proposition \ref{1.20}.
Let
\[
M_1^*(X)
\]
denote its cocone.
Then there is a distinguished triangle
\begin{equation}
\label{1.15.1}
M_1^*(X)\rightarrow \underline{\rm Hom}_\DMeffk(M_{\geq 1}(X),\bZ(1)[2])\rightarrow {\rm NS}(X)\rightarrow M_1^*(X)[1].
\end{equation}
\end{df}

\begin{df}
\label{1.35}
Let $G$ be a group scheme over $k$.
We say that $G$ is an {\it abelian group scheme} (resp.\ {\it semi-abelian group scheme}) if
the connected component of the identity $G^0$ is an abelian (resp.\ semi-abelian) variety and $G/G^0$ is a finitely generated abelian group.
\end{df}

\begin{prop}\label{1.36}
Let $G$ be a semi-abelian group scheme, and let $L$ be a lattice.
Regard them as Nisnevich sheaves on $Sm/k$.
Then $G\oplus L$ has a unique transfer structure.
\end{prop}
\begin{proof}
Set $F:=G\oplus L$.
In \cite{MR2102056}, it is shown that $G$ has a transfer structure.
Then $F$ has a transfer structure too.
For uniqueness, suppose $X,Y\in Sm/k$ and $Z\in {\rm Cor}(X,Y)$.
If $F$ has another transfer structure, then the two transfer structures and $Z$ induce homomorphisms
\[
p\colon F(Y)\rightarrow F(X)
\text{ and }
q\colon F(Y)\rightarrow F(X).\]
We need to show $p=q$.

Morphisms in $Sm/k$ are determined by closed points, so the homomorphism
\[
F(X)\rightarrow \coprod_{x\in {\rm cl}\,X}F(\{x\})
\]
induced by the closed immersions $\{x\}\rightarrow X$ is injective, where ${\rm cl}\,X$ denotes the set of closed points of $X$.
By composing $Z$ with $\{x\}\rightarrow X$ for all $x\in {\rm cl}\,X$, we reduce to the case when $X=\Spec(k)$.
In this case, there is an isomorphism
\[
{\rm Cor}({\rm Spec}\,k,Y)\simeq \bZ[Y(k)].
\]
Thus $Z$ can be written as a formal sum of morphisms $\Spec(k)\rightarrow Y$.
Since the two transfer structures agree on the level of morphisms, $p$ and $q$ are equal.
\end{proof}

\begin{rmk}
\label{1.37}
In this paper, we view any scheme $G$ as a representable sheaf on $Sm/k$.
Then $G$ and $G_{red}$ represent the same sheaf on $Sm/k$.
Indeed, if $T\in Sm/k$, then every morphism $f\colon T\rightarrow G$ canonically factors through $G_{red}$ since $T$ is reduced.
\end{rmk}
\begin{prop}
\label{1.14}
For every proper $X\in Sm/k$, the Nisnevich sheaf with transfers ${\rm Pic}_{X/k}$ is  representable by an abelian group scheme.
\end{prop}
\begin{proof}
Let $A$ be the presheaf of abelian groups on $Sch/k$ given by
\[
A(T):={\rm Pic}(X\times T)/{\rm Pic}(T)
\]
for $T\in Sch/k$.
The restriction of $A$ to $Sm/k$ is ${\rm Pic}_{X/k}$.
It is well-known that $A$ is representable by a scheme $G$ such that $G_{red}$ is an abelian group scheme.
Due to Remark \ref{1.37}, $\Pic_{X/k}$ is as Nisnevich sheaves representable by the abelian group scheme $G_{red}$.
Moreover, the transfer structure on $\Pic_{X/k}$ discussed in Definition \ref{1.12} agrees with that in \cite{MR2102056} by Proposition \ref{1.36}.
\end{proof}

\begin{prop}
\label{1.38}
For every proper simplicial scheme $X_\bullet$ in $Sm/k$, the Nisnevich sheaf with transfers $\Pic_{X_\bullet/k}$ is representable by a semi-abelian scheme group scheme.
\end{prop}
\begin{proof}
The proof is similar to that of Proposition \ref{1.14}.
Just add the reference \cite[Theorem 5.1.1]{MR2107442} saying that $\Pic_{X_\bullet/k}$ is representable by a scheme $G$ such that $G_{red}$ is a semi-abelian group scheme.
\end{proof}
\section{Hilbert's Theorem 90}
\begin{none}\label{1.8}
If $X$ is not proper over $k$, then ${\rm Pic}_{X/k}^0$ is not representable by an abelian variety.
Instead, a little improved version of \cite[Proposition 3.5.1]{MR3545132} is that there is an exact sequence
\begin{equation}\label{1.8.1}
L\rightarrow A\rightarrow {\rm Pic}_{X/k}^0\rightarrow 0
\end{equation}
of {\it \'etale} sheaves, where $L$ is a lattice, and $A$ is an abelian variety.
We need its Nisnevich version, which asserts that there is an exact sequence \eqref{1.8.1} of {\it Nisnevich} sheaves. Note that the Nisnevich version is not a consequence of the \'etale version. Indeed, for any isogeny $B\rightarrow A$ of abelian varieties, there is an exact sequence
\[
L'\rightarrow B\rightarrow {\rm Pic}_{X/k}^0\rightarrow 0
\]
of \'etale sheaves such that $L'$ is a lattice. Thus there are lots of possible choices of $A$, and we need to choose one of them fitting in \eqref{1.8.1} for the Nisnevich topology.
We do this by explicitly constructing $L$ and $A$. For this purpose, we use de Jong's alterations to reduce to the case when $X$ is smooth and projective over $k$.
\end{none}

\begin{none}
\label{1.9}
Let us begin with recalling Hilbert's theorem 90 in \'etale cohomology theory. It asserts that the induced homomorphism of cohomology groups
\[
H_{Zar}^1(X,\bG_m)\rightarrow H_{\et}^1(X,\bG_m)
\]
is an isomorphism for every $X\in Sch/k$.
The same proof works if we replace {\it Zar} by {\it Nis}, so the induced homomorphism of cohomology groups
\[
H_{Nis}^1(X,\bG_m)\rightarrow H_{\et}^1(X,\bG_m)
\]
is an isomorphism.

There is also a simplicial version. In \cite[Proposition 4.4.1]{MR1891270}, it is shown that the induced homomorphism of hypercohomology groups
\[
\bH_{Zar}^1(X_\bullet,\bG_m)\rightarrow \bH_{\et}^1(X_\bullet,\bG_m)
\]
is an isomorphism for every simplicial scheme $X_\bullet$ over $k$.
As above, the induced homomorphism of hypercohomology groups
\begin{equation}
\label{1.9.1}
\bH_{Nis}^1(X_\bullet,\bG_m)\rightarrow \bH_{\et}^1(X_\bullet,\bG_m)
\end{equation}
is an isomorphism by the same proof.

Now we apply this to study the structure of $\underline{\rm Hom}_\DMeffk(M(X),\bZ(1)[2])$ as follows.
We refer to \cite[Definition 3.1.2]{zbMATH00912445} for the definition of the $h$-topology.
\end{none}
\begin{prop}\label{1.6}
For every $h$-hypercover $p\colon X_\bullet\rightarrow X$ in $Sm/k$, the induced morphisms
  \[\tau_{\geq 0}\underline{\rm Hom}_{\DMeffk}(M(X),\bZ(1)[2])\rightarrow \underline{\rm Hom}_{\DMeffk}(M(X),\bZ(1)[2]),\]
  \[\tau_{\geq 0}\underline{\rm Hom}_{\DMeffk}(M(X),\bZ(1)[2])\rightarrow \tau_{\geq 0}\underline{\rm Hom}_{\DMeffk}(M(X_\bullet),\bZ(1)[2])\]
are isomorphisms.
\end{prop}
\begin{proof}
Let $f\colon X\rightarrow k$ be the structure morphism. By \cite[Proposition 3.2.8]{MR1764202}, we need to show that the induced morphisms
\begin{gather*}
\tau_{\geq -1}Rf_*f^*\bG_m\rightarrow Rf_*f^*\bG_m,
\\
\tau_{\geq -1}Rf_*f^*\bG_m\rightarrow \tau_{\geq -1}Rf_*Rp_*p^*f^*\bG_m
\end{gather*}
in the derived category ${\rm D}(k_{Nis},\bZ)$ of Nisnevich sheaves of abelian groups on the small Nisnevich site $k_{Nis}$ are isomorphisms, where $\tau_{\geq i}$ denotes the usual truncation functor.
In other words, we need to show that
  \begin{enumerate}
    \item[(1)] $R^if_*f^*\bG_m=0$ for $i\neq 0,1$,
    \item[(2)] the induced morphism $R^if_*f^*\bG_m\rightarrow R^i(fp)_*(fp)^*\bG_m$ is an isomorphism for $i=0,1$.
  \end{enumerate}
We reduce to showing that for every local henselian scheme $Y$ over $k$,
  \begin{enumerate}
    \item[(1)'] $H_{Nis}^i(Y\times X,\bG_m)=0$ for $i\neq 0,1$,
    \item[(2)'] the induced homomorphism $H_{Nis}^i(Y\times X,\bG_m)\rightarrow \bH_{Nis}^i(Y\times X_\bullet,\bG_m)$ is an isomorphism for $i=0,1$.
  \end{enumerate}
By \cite[Vanishing Theorem 19.3]{MVW}, $H_{Nis}^i(T,\bG_m)=0$ for every $i\neq 0,1$ and $T\in Sm/k$.
This implies (1)' since $H_{Nis}^i(-,\bG_m)$ commutes with filtered limits by \cite[Corollarie VI.8.7.7]{SGA4}.
Hence it remains to show (2)'.

Due to Hilbert's Theorem 90 \eqref{1.9.1}, it suffices to show that the induced homomorphism
  \[H_{\et}^i(Y\times X,\bG_m)\rightarrow \bH_{\et}^i(Y\times X_\bullet,\bG_m)\]
  is an isomorphism for every integer $i$.
By applying \cite[Corollarie VI.8.7.7]{SGA4} to Proposition \ref{1.7} below we conclude.
\end{proof}
\begin{prop}\label{1.7}
Suppose $X\in Sm/k$ and $p\colon X_\bullet\rightarrow X$ is an $h$-hypercover.
Then the induced homomorphism
  \[H_{\et}^i(X,\bG_m)\rightarrow \bH_{\et}^i(X_\bullet,\bG_m)\]
is an isomorphism for every integer $i$.
\end{prop}
\begin{proof}
  Let $f\colon X\rightarrow k$ be the structure morphism, and let $K$ be a cone of the induced morphism
  \[R_{\et}f_*f^*\bG_m\rightarrow R_{\et}f_*R_{\et}p_*p^*f^*\bG_m\]
  in the derived category ${\rm D}(k_{\et},\bZ)$ of \'etale sheaves of abelian groups on the small \'etale site $k_{\et}$. We only need to show the vanishing
  \[{\rm Hom}_{{\rm D}(k_{\et},\bZ)}(\bZ,K[i])=0\]
  for every integer $i$ since the category of \'etale sheaves of abelian groups on $k_{\et}$ is equivalent to the category of $k$-modules.

  By \cite[Theorems 16.1.3, 16.2.18, Corollary 16.2.22]{CD19}, the two motives $M(X)$ and $M(X_\bullet)$ are isomorphic in ${\rm DM}_{\et}(k,{\bf Q})$.
Together with \cite[Corollary 1.8.5]{MR3545132}, we have isomorphisms
  \[\begin{split}
    H_{\et}^i(X,\bG_m\otimes {\bf Q})&\simeq {\rm Hom}_{{\rm DM}_{\et}(k,{\bf Q})}(M(X),\bG_m[i])\\
    &\simeq {\rm Hom}_{{\rm DM}_{\et}(k,{\bf Q})}(M(X_\bullet),\bG_m[i])\\
    &\simeq H_{\et}^i(X_\bullet,\bG_m\otimes {\bf Q}).
  \end{split}\]
It follows that for every integer $i$, we have the vanishing
\[{\rm Hom}_{{\rm D}(k_{\et},\bZ)}(\bZ,K\otimes {\bf Q}[i])=0.\]
By \cite[Corollarie VI.5.3]{SGA4}, this becomes
\[
{\rm Hom}_{{\rm D}(k_{\et},\bZ)}(\bZ,K[i])\otimes {\bf Q}=0.
\]
Hence ${\rm Hom}_{{\rm D}(k_{\et},\bZ)}(\bZ,K[i])$ is a torsion abelian group.
It remains to show that for every integer $i$ and prime $n$ we have the vanishing
\begin{equation}
\label{1.7.1}
{\rm Hom}_{{\rm D}(k_{\et},\bZ)}(\bZ,K[i])\otimes \bZ/n=0.
\end{equation}

There are exact sequences of \'etale sheaves
\[0\rightarrow \bG_m\stackrel{\cdot n}\rightarrow \bG_m\rightarrow \bG_m/\bG_m^n\rightarrow 0\quad \text{if }\,{\rm char}\,k\,\vert \,n,\]
\[0\rightarrow \mu_n\rightarrow \bG_m\stackrel{\cdot n}\rightarrow \bG_m\rightarrow 0\quad \text{if }\,{\rm char}\,k\nmid n.\]
Set $F:=\bG_m/\bG_m^n$ if ${\rm char}\,k\,\lvert\,n$ and $F:=\mu_n$ if ${\rm char}\,k\nmid n$. Then $F$ is an \'etale sheaf of $\bZ/n$-modules. In the proof of \cite[Proposition 5.3.3]{MR3477640}, it is shown that any $h$-cover is universally of cohomological descent with respect to the fibered category of \'etale sheaves of $\bZ/n$-modules. By \cite[Theorem 7.10]{Con}, any $h$-hypercover is universally of cohomological descent. This implies that  for every integer $i$ the induced homomorphism
  \[H_{\et}^i(X,F)\rightarrow \bH_{\et}^i(X_\bullet,F)\]
is an isomorphism.
Thus for every integer $i$, we have the vanishing
\[
{\rm Hom}_{{\rm D}(k_{\et},\bZ)}(\bZ,K\otimes \bZ/n[i])=0.
\]
This implies the vanishing \eqref{1.7.1}.
\end{proof}

\section{\texorpdfstring{$0$}{0}-motivic sheaves}
\begin{df}
  A constant sheaf on $Sm/k$ is called a $0$-{\it motivic sheaf}. We denote by ${\rm Sh}^{0-{\rm mot}}(k)$ the category of $0$-motivic sheaves.
\end{df}
\begin{rmk}\label{1.27}
  The definition of $0$-motivic sheaves is simpler than that in \cite{MR2494373} since we assume that $k$ is algebraically closed.
\end{rmk}
\begin{none}
\label{1.28}
Recall from \cite[Corollary 1.2.5]{MR2494373} that a canonical exact fully faithful functor
\[
\sigma_0\colon {\rm Sh}^{0-{\rm mot}}(k)\rightarrow {\rm Sh}^{tr}(k)
\]
admits a left adjoint
\[
\pi_0\colon {\rm Sh}^{tr}(k)\rightarrow {\rm Sh}^{0-{\rm mot}}(k)
\]
such that $\pi_0(\bZ^{tr}(X))\simeq \bZ^r$ if $X$ has $r$ connected components.
Observe that we have the vanishing
\begin{equation}
\label{1.28.2}
\pi_0(\bZ_{\geq 1}^{tr}(X))=0.
\end{equation}
  
Moreover, \cite[Corollary 2.3.3]{MR2494373} gives a functor
\[
L\pi_0\colon \rD(\Sh^{tr}(k))\rightarrow \rD(\Sh^{0-{\rm mot}}(k))
\]
that is left adjoint to a canonical fully faithful functor
\[
\sigma_0\colon \rD(\Sh^{0-{\rm mot}}(k))\rightarrow \rD(\Sh^{tr}(k)).
\]
Since $\sigma_0$ is fully faithful, for every abelian group $N$ we have isomorphisms
\begin{equation}
\label{1.28.1}
L\pi_0(N)\simeq (L\pi_0\circ \sigma_0)(N)\simeq N.
\end{equation}
\end{none}

\begin{rmk}
If $G$ is a semi-abelian group scheme, then $\pi_0(G)$ has two different meanings.
The first one is the set of connected components of $\pi_0(G)$.
The second one is in \ref{1.28}.
We will usually talk about the second one.
If we need to talk about the first one, we will write $\pi_0'(G)$ to avoid confusion.
\end{rmk}

\begin{prop}\label{1.24}
Suppose $X\in Sm/k$ and $F$ is a constant sheaf on $Sm/k$.
Then we have the vanishing
  \[\underline{\rm Hom}_\DMeffk(M_{\geq 1}(X),F)=0.\]
\end{prop}
\begin{proof}
  We may assume that $X$ is integral.
  For every $T\in Sm/k$ and integer $i$, it suffices to show that the induced homomorphism
  \[{\rm Hom}_\DMeffk(M(T),F[i])\rightarrow {\rm Hom}_\DMeffk(M(T\times X),F[i])\]
  is an isomorphism. Hence we just need to prove that the induced homomorphism
  \begin{equation}\label{1.24.1}
  H_{Nis}^i(k,F)\rightarrow H_{Nis}^i(X,F)
  \end{equation}
  is an isomorphism.

  Assume $F=\bZ$.
  Then by \cite[Vanishing Theorem 19.3]{MVW}, $H_{Nis}^i(X,F)=0$ for $i>0$. For $i=0$, \eqref{1.24.1} is an isomorphism since $\bZ$ is constant. Thus we are done for this case. We also have that \eqref{1.24.1} is an isomorphism if $F$ is a lattice.

  Now we treat the general case. Since $F$ is a colimit of lattices, there is an exact sequences
  \[0\rightarrow F''\rightarrow F'\rightarrow F\rightarrow 0\]
  of constant sheaves on $Sm/k$ such that $F'$ and $F''$ are direct sums of lattices. Then $F'$ and $F''$ are filtered colimit of lattices. Since $H_{Nis}^i(X,-)$ commutes with filtered colimits by \cite[Corollarie VI.5.3]{SGA4}, the induced homomorphisms
  \[H_{Nis}^i(k,F')\rightarrow H_{Nis}^i(X,F'),\quad H_{Nis}^i(k,F'')\rightarrow H_{Nis}^i(X,F'')\]
  are isomorphisms as above. Thus \eqref{1.24.1} is an isomorphism by the five lemma.
\end{proof}
\begin{prop}\label{1.21}
For every $X\in Sm/k$, there is an isomorphism
\[
\pi_0(\Pic_{X/k})\simeq \NS(X).
\]
\end{prop}
\begin{proof}
There is an exact sequence
\[
0\rightarrow \Pic_{X/k}^0\rightarrow \Pic_{X/k}\rightarrow \NS(X)\rightarrow 0.
\]
Since $\pi_0$ is left adjoint to the inclusion functor, $\pi_0$ is right exact.
Thus there is an exact sequence
\[
\pi_0(\Pic_{X/k}^0)\rightarrow \pi_0(\Pic_{X/k})\rightarrow \pi_0(\NS(X))\rightarrow 0.
\]
Since $\NS(X)$ is constant, $\pi_0(\NS(X))\simeq \NS(X)$.
Hence it suffices to show that we have the vanishing $\pi_0({\rm Pic}_{X/k}^0)=0$.

Since $\pi_0({\rm Pic}_{X/k}^0)$ is constant, we only need to show the vanishing
\[
\pi_0({\rm Pic}_{X/k}^0)(k)=0.
\]
Let $a$ be an element of $\pi_0({\rm Pic}_{X/k}^0)(k)$, and let
\[
Z\in {\rm Pic}_{X/k}^0(k)={\rm Pic}^0(X)
\]
be an element whose image in $\pi_0({\rm Pic}_{X/k}^0)(k)$ is $a$.
By definition, $Z$ is algebraically equivalent to $0$.
Thus there is a smooth and separated curve $C$ and an element $W$ of ${\rm Pic}_{X/k}^0(C)$ such that the pullbacks of $W$ to some closed points $x_1$ and $x_2$ are $Z$ and $0$. Then $W$ gives a morphism $\bZ^{tr}(C)\rightarrow {\rm Pic}_{X/k}^0$ of Nisnevich sheaves with transfers. The composite morphism
\[
\bZ\rightarrow \bZ^{tr}(C)\rightarrow {\rm Pic}_{X/k}^0\]
is zero, where the first arrow is induced by $x_2$.
Thus this induces a morphism
\[
\bZ_{\geq 1}^{tr}(C)\rightarrow {\rm Pic}_{X/k}^0
\]
of Nisnevich sheaves with transfers.
Note that $Z$ is in its image of
\[
\bZ_{\geq 1}^{tr}(C)(k)\rightarrow {\rm Pic}_{X/k}^0(k).
\]
Taking $\pi_0$, we get a homomorphism
\[
\pi_0(\bZ_{\geq 1}^{tr}(C))(k)\rightarrow \pi_0({\rm Pic}_{X/k}^0)(k).
\]
Then $a$ is in its image.
By \eqref{1.28.2}, $\pi_0(\bZ_{\geq 1}^{tr}(C))=0$. Thus $a=0$.
\end{proof}
\begin{prop}\label{1.29}
  Let $G$ be a semi-abelian variety, and consider the functor $\pi_0$ in {\rm \ref{1.28}}. Then $\pi_0(G)=0$.
\end{prop}
\begin{proof}
Consider a canonical exact sequence
\[
0\rightarrow L\otimes \bG_m\rightarrow G\rightarrow A\rightarrow 0
\]
of Nisnevich sheaves with transfers, where $L$ is a lattice, and $A$ is an abelian variety. Since $\pi_0$ is left adjoint to the inclusion functor, it is right exact. Thus there is an exact sequence
\begin{equation}
\label{1.29.1}
\pi_0(L\otimes \bG_m)\rightarrow \pi_0(G)\rightarrow \pi_0(A)\rightarrow 0.
\end{equation}
Since $\bG_m\simeq \bZ_{\geq 1}^{tr}(\bG_m)$ in $\DMeffk$, \eqref{1.28.2} gives $\pi_0(\bG_m)=0$.
By Proposition \ref{1.21},
\[
\pi_0(A)\simeq \pi_0({\rm Pic}_{A^\vee/k}^0)= 0.
\]
Thus we have the vanishing $\pi_0(G)=0$ from \eqref{1.29.1}.
\end{proof}
\begin{prop}\label{1.33}
Let $G$ be a semi-abelian variety. Then for every integer $i\leq 1$, we have the vanishing
\[
{\rm Hom}_\DMeffk(G,\bZ[i])=0.
\]
\end{prop}
\begin{proof}
Since $G$ and $\bZ$ are in the heart of the $0$-motivic $t$-structure, we are done when $i<0$.
For $i=0$, we are done by Proposition \ref{1.29}.
Hence it remains to deal with the case $i=1$.
  
In \cite[Paragraph 2.1.5]{zbMATH03815835} there is an exact sequence
\begin{equation}\label{1.33.1}
\bZ[G\times G\times G]\oplus \bZ[G\times G]\stackrel{v}\rightarrow \bZ[G\times G]\stackrel{u}\rightarrow \bZ[G]\rightarrow G\rightarrow 0
\end{equation}
of Nisnevich sheaves (without transfers).
Here, $u$ and $v$ are given by the formula
\begin{equation}
\label{1.33.3}
\begin{split}
u([x,y])=[x]&+[y]-[x+y],
\\
v([x,y,z],[p,q])=[y,z]-[x+y,z]&+[x,y+z]-[x,y]+[p,q]-[q,p].
\end{split}
\end{equation}
Let $\Sh(Sm/k,\bZ)$ denote the category of Nisnevich sheaves of abelian groups on $Sm/k$. The induced sequence
\begin{equation}\label{1.33.2}
\begin{split}
&{\rm Hom}_{\Sh_{Nis}(Sm/k,\bZ)}(\bZ[G],\bZ)\rightarrow {\rm Hom}_{\Sh_{Nis}(Sm/k,\bZ)}(\bZ[G\times G],\bZ)\\
\rightarrow &{\rm Hom}_{\Sh_{Nis}(Sm/k,\bZ)}(\bZ[G\times G\times G]\oplus \bZ[G\times G],\bZ)
\end{split}
\end{equation}
is isomorphic to the induced sequence
\[
\begin{split}
&\bZ[{\rm Hom}_{Sm/k}(G,{\rm Spec}\,k)]\rightarrow \bZ[{\rm Hom}_{Sm/k}(G\times G,{\rm Spec}\,k)]
\\
\rightarrow &\bZ[{\rm Hom}_{Sm/k}(G\times G\times G,{\rm Spec}\,k)]\oplus \bZ[{\rm Hom}_{Sm/k}(G\times G,{\rm Spec}\,k)].
\end{split}
\]
With the formula \eqref{1.33.3} we see that this sequence is isomorphic to
\[
\bZ\stackrel{{\rm id}}\rightarrow \bZ\stackrel{0}\rightarrow \bZ\oplus \bZ.
\]
Note that this is an exact sequence.
Moreover, for every $X\in Sm/k$ and integer $i>0$, we have the vanishing
\[
{\rm Ext}_{{\rm Sh}_{Nis}(Sm/k,\bZ)}^i(\bZ[X],\bZ)\simeq H_{Nis}^i(X,\bZ)=0.
\]
Thus by \eqref{1.33.1} and the exactness of \eqref{1.33.2}, we have the vanishing
\begin{equation}
\label{1.33.4}
{\rm Ext}_{{\rm Sh}_{Nis}(Sm/k,\bZ)}^1(G,\bZ)=0.
\end{equation}

Suppose that
\[
0\rightarrow \bZ\rightarrow F\rightarrow G\rightarrow 0
\]
is an exact sequence of Nisnevich sheaves with transfers.
The vanishing \eqref{1.33.4} shows that $F\simeq \bZ\oplus G$ in the category of Nisnevich sheaves.
Then by Proposition \ref{1.36}, $F\simeq \bZ\oplus G$ in the category of Nisnevich sheaves with transfers.
It follows that we have the vanishing
\[
{\rm Ext}_{{\rm Sh}^{tr}(k)}^1(G,\bZ)=0.
\]
To finish the proof observe that $\bZ$ is $\bA^1$-local.
\end{proof}

\begin{prop}\label{2.13}
Let $G$ be a semi-abelian variety.
Then the motive
\[
\underline{\rm Hom}_\DMeffk(\bG_m,G)
\]
is a lattice.
\end{prop}
\begin{proof}
There is a distinguished triangle
\[
L\otimes \bG_m\rightarrow G\rightarrow A\rightarrow L\otimes \bG_m[1],
\]
where $L$ is a lattice, and $A$ is an abelian variety. By the cancellation theorem \cite{MR2804268},
there is an isomorphism
\[
\underline{\rm Hom}_\DMeffk(\bG_m,L\otimes \bG_m)\simeq L.
\]
Thus it suffices to show the vanishing
\begin{equation}
\label{2.13.1}
\underline{\rm Hom}_\DMeffk(\bG_m,A)=0.
\end{equation}

By Proposition \ref{1.25} there is a distinguished triangle
\[
A\to \uHom_{\DMeffk}(M_{\geq 1}(A^\vee),\bZ(1)[2])\to \NS(A^\vee)\to A[1].
\]
Proposition \ref{1.24} gives the vanishing
\[
\uHom_{\DMeffk}(\bG_m,\NS(A^\vee))=0.
\]
Hence it suffices to show that
\begin{align*}
&\uHom_{\DMeffk}(\bG_m,\uHom_{\DMeffk}(M_{\geq 1}(A^\vee),\bZ(1)[2]))
\\
\simeq &
\uHom_{\DMeffk}(\bG_m\otimes M_{\geq 1}(A^\vee),\bZ(1)[2])
\\
\simeq &
\uHom_{\DMeffk}(M_{\geq 1}(A^\vee),\bZ[1])
\end{align*}
is vanishing, where the second isomorphism is given by the cancellation theorem \cite{MR2804268}.
Apply Proposition \ref{1.24} again to conclude.
\end{proof}

\section{Structure of \texorpdfstring{$M_1^*(X)$}{M1*(X)}}
The purpose of this section is to construct an isomorphism
\[
M_1^*(X)[-1]\simeq[L\rightarrow B],
\]
where $L$ is a lattice, and $B$ is a semi-abelian variety.

\begin{df}\label{1.23}
If $U\to X$ be an open immersion in $Sm/k$, we set
\[
M(X/U):=C_*(\bZ^{tr}(X)/\bZ^{tr}(U)),
\]
and consider it as an object of $\DMeffk$.
\end{df}

\begin{prop}
\label{2.49}
Let $j\colon U\to X$ be an open immersion of integral schemes in $Sm/k$.
Then there is an isomorphism
\[
\uHom_\DMeffk(M(X/U),\bZ(1)[2])
\simeq
\bZ^s
\]
in $\DMeffk$, where $s$ is the number of irreducible components in $X-U$ whose dimensions are equal to $\dim X-1$.
\end{prop}
\begin{proof}
For every integer $i\geq 0$ and $X\in Sm/k$, let $z^i(X,*)$ be the cycle complex in \cite[Definition 17.1]{MVW}.
According to \cite[Theorem 19.8]{MVW}, there is a quasi-isomorphism
\[
\bZ(1)[2]\xrightarrow{\sim} z^i(-\times \bA^1,*)
\]
of complexes of Zariski sheaves.
This gives a canonical isomorphism
\[
\uHom_\DMeffk(M(Y),\bZ(1)[2])\xrightarrow{\sim} z^1(-\times Y\times \bA^1,*)
\]
in $\DMeffk$ for $Y\in Sm/k$.
Hence $\uHom_\DMeffk(M(X/U),\bZ(1)[2])$ can be identified with a cone of
\[
z^1(-\times X\times \bA^1,*)\to z^1(-\times U\times \bA^1,*).
\]
Together with Bloch's localization sequence \cite{zbMATH00653321} and \cite[Proposition 19.12]{MVW}, we deduce that for every integral $T\in Sm/k$ and integer $i$ there is an isomorphism
\begin{equation}
\label{2.49.1}
{\rm Hom}_{\DMeffk}(M(T),\uHom_\DMeffk(M(X/U),\bZ(1)[2])[-i])
\simeq
CH_{n-1+r}(Z\times T,i),
\end{equation}
where $Z:=X-U$, $n:=\dim X$, and $r:=\dim T$.
We also have an isomorphism
\begin{equation}
\label{2.49.2}
CH_{n-1+r}(Z\times T,i)
\simeq
\left\{
\begin{array}{cc}
\bZ^s & \text{if }i=0,
\\
0 & \text{oterwise}
\end{array}
\right.
\end{equation}
since $s$ the number of irreducible components in $Z\times T$ whose dimensions are equal to $n-1+r$.
Combine \eqref{2.49.1} and \eqref{2.49.2} to have a desired isomorphism.
\end{proof}

\begin{lem}
\label{1.40}
Any subsheaf $\cG$ of a constant sheaf $\cF$ on $Sm/k$ is again constant.
In particular, any subsheaf of a lattice is again a lattice.
\end{lem}
\begin{proof}
Suppose $\cF$ is associated with a set $T$.
For every integral $X\in Sm/k$, there is a rational point of $X$.
Hence there is a commutative diagram
\[
\begin{tikzcd}
\cG(k)\ar[rd,hookrightarrow]\ar[r,"b"]&
\cG(X)\ar[d,hookrightarrow]\ar[r,"a"]&
\cG(k)\ar[ld,hookrightarrow,"c"]
\\
&T.
\end{tikzcd}
\]
Since $ab$ is bijective, $a$ is surjective.
Since $ca$ is injective, $a$ is injective.
Hence $a$ is bijective.
\end{proof}

\begin{prop}\label{1.10}
For every $X\in Sm/k$, there is an isomorphism
\begin{equation}
\label{1.10.1}
\underline{\rm Hom}_\DMeffk(M_{\geq 1}(X),\bZ(1)[2])[-1]
\simeq
[L\rightarrow B]
\end{equation}
in $\DMeffk$ for some lattice $L$ and homotopy invariant Nisnevich sheaf with transfers $B$ such that $\ker(B\to \pi_0(B))$ is a semi-abelian variety,
\end{prop}
\begin{proof}
We may assume that $X$ is integral by constructing \eqref{1.10.1} for every connected component.
Due to \cite[Section 1]{MR1423020}, there exists a diagram
\[
X\xleftarrow{p} X_\bullet \xrightarrow{j} \overline{X}_\bullet
\]
of simplicial schemes in $Sm/k$, where each $\overline{X}_i$ is proper over $k$, $p$ is an $h$-hypercover, and each $X_i\to \overline{X}_i$ is an open immersion whose reduced complement is a strict normal crossing divisor,

There is a canonical distinguished triangle
\[
M(X_\bullet)\rightarrow M(\overline{X}_\bullet)\rightarrow M(\overline{X}_\bullet/X_\bullet)\rightarrow M(X_\bullet)[1].
\]
This induces a canonical distinguished triangle
\begin{equation}
\label{1.10.2}
\begin{split}
&\uHom_\DMeffk(M(\overline{X}_\bullet/X_\bullet),\bZ(1)[2])
\rightarrow
\uHom_\DMeffk(M(\overline{X}_\bullet),\bZ(1)[2])
\\
\rightarrow &
\uHom_\DMeffk(M(X_\bullet),\bZ(1)[2])
\rightarrow
\uHom_\DMeffk(M(\overline{X}_\bullet/X_\bullet),\bZ(1)[2])[1].
\end{split}
\end{equation}
Proposition \ref{2.49} gives an isomorphism
\[
\underline{\rm Hom}_\DMeffk(M(\overline{X}_i/X_i),\bZ(1)[2])\simeq \bZ^{d_i},
\]
where $d_i$ is the number of irreducible components of $\overline{X}_i-X_i$.
We set
\[
F:=\uHom_\DMeffk(M(\overline{X}_\bullet),\bZ(1)[2])
\]
for simplicity.
Use the homotopy $t$-structure to \eqref{1.10.2} to find a distinguished triangle
\begin{equation}
\label{1.10.3}
[\bZ^{d_0}\rightarrow N]
\rightarrow
\tau_{\geq 0}F
\rightarrow
\tau_{\geq 0}\underline{\rm Hom}_\DMeffk(M(X_\bullet),\bZ(1)[2])
\rightarrow
[\bZ^{d_0}\rightarrow N][1],
\end{equation}
where
\[
N:=\ker(\ker(\bZ^{d_1}\rightarrow \bZ^{d_2})\rightarrow h_{-1}F).
\]
Due to Lemma \ref{1.40}, $N$ is a lattice.

By Proposition \ref{1.6}, there is an isomorphism
\begin{equation}
\label{1.10.4}
\tau_{\geq 0}\underline{\rm Hom}_\DMeffk(M(X_\bullet),\bZ(1)[2])
\simeq
\uHom_\DMeffk(M(X),\bZ(1)[2]).
\end{equation}
The kernel of
\[
\bZ^{\pi_0(X_0)}\rightarrow \bZ^{\pi_0(X_1)}
\]
that is the dual of the induced morphism $\bZ^{\pi_0(X_1)}\rightarrow \bZ^{\pi_0(X_0)}$ is isomorphic to $\bZ^{\pi_0(X_0)}\simeq \bZ$.
The same is also true for the kernel of
\[
\bZ^{\pi_0(\overline{X}_0)}\rightarrow \bZ^{\pi_0(\overline{X}_1).}
\]
Apply Proposition \ref{1.39} to the distinguished triangle
\[
\tau_{\geq 1}F\rightarrow \tau_{\geq 0}F\rightarrow h_0F\rightarrow \tau_{\geq 1}F[1]
\]
to obtain a distinguished triangle
\[
\bZ(1)[2]
\rightarrow
\tau_{\geq 0}F
\rightarrow
\Pic_{\overline{X}_\bullet/k}
\rightarrow
\bZ(1)[3].
\]
There is a commutative diagram of solid arrows
\[
\begin{tikzcd}[column sep=small, row sep=small]
0\arrow[d]\arrow[r]&
\bZ(1)[2]\arrow[r,"\id"]\arrow[d]&
\bZ(1)[2]\arrow[d]\arrow[r]&
0\arrow[d]
\\
{[\bZ^{d_0}\rightarrow N]}\arrow[d,"\id"']\arrow[r]&
\tau_{\geq 0}F\arrow[d]\arrow[r]&
\underline{\rm Hom}_\DMeffk(M(X),\bZ(1)[2])\arrow[d]\arrow[r]&
{[\bZ^{d_0}\rightarrow N][1]}\arrow[d,"\id"]
\\
{[\bZ^{d_0}\rightarrow N]}\arrow[r,dashrightarrow]\arrow[d]&
\Pic_{\overline{X}_\bullet/k}\arrow[d]\arrow[r,dashrightarrow]&
\uHom_\DMeffk(M_{\geq 1}(X),\bZ(1)[2])\arrow[d]\arrow[r,dashrightarrow]&
{[\bZ^{d_0}\rightarrow N]}[1]\arrow[d]
\\
0\arrow[r]&
\bZ(1)[3]\arrow[r,"\id"]&
\bZ(1)[3]\arrow[r]&
0,
\end{tikzcd}
\]
where the second row is obtained by \eqref{1.10.3} and \eqref{1.10.4}.
By the octahedral axiom formulated as in \cite[Proposition 1.4.6]{zbMATH01573275}, this can be completed into a commutative diagram containing all the dotted arrows such that every column and row is a distinguished triangle.
Let $B$ be a cocone of the composite morphism
\[
\uHom_\DMeffk(M_{\geq 1}(X),\bZ(1)[2])
\to
[\bZ^{d_0}\to N][1]
\to
\bZ^{d_0}[1].
\]
Then there is an isomorphism
\begin{equation}
\label{1.10.5}
\uHom_\DMeffk(M_{\geq 1}(X),\bZ(1)[2])[-1]
\simeq
[\bZ^{d_0}\to B].
\end{equation}
There is a commutative diagram of solid arrows
\[
\begin{tikzcd}[column sep=tiny,row sep=small]
{[\bZ^{d_0}\rightarrow N][-1]}\arrow[r]\arrow[d]&
{\rm Pic}_{\overline{X}_\bullet/k}[-1]\arrow[r]\arrow[d]&
\underline{\rm Hom}_\DMeffk(M_{\geq 1}(X),\bZ(1)[1])\arrow[r]\arrow[d]&
{[\bZ^{d_0}\rightarrow N]}\arrow[d]
\\
\bZ^{d_0}[-1]\arrow[r]\arrow[d]&
0\arrow[r]\arrow[d]&
\bZ^{d_0}\arrow[r,"{\rm id}"]\arrow[d]&
\bZ^{d_0}\arrow[d]
\\
N[-1]\arrow[r,dashrightarrow]\arrow[d]&
{\rm Pic}_{\overline{X}_\bullet/k}\arrow[r,dashrightarrow]\arrow[d,"{\rm id}"]&
B\arrow[r,dashrightarrow]\arrow[d]&
N\arrow[d]
\\
{[\bZ^{d_0}\rightarrow N]}\arrow[r]&
{\rm Pic}_{\overline{X}_\bullet/k}\arrow[r]&
\underline{\rm Hom}_\DMeffk(M_{\geq 1}(X),\bZ(1)[2])\arrow[r]&
{[\bZ^{d_0}\rightarrow N][1]}.
\end{tikzcd}
\]
By the octahedral axiom again, this can be completed into a commutative diagram containing all the dotted arrows such that every column and row is a distinguished triangle.

Due to \eqref{1.10.5}, it remains to show that $B$ is a semi-abelian group scheme.
From the third row, we have an exact sequence
\[
0\rightarrow \Pic_{\overline{X}_\bullet/k}\rightarrow B\rightarrow N\rightarrow 0.
\]
Due to \eqref{1.28.1}, $h_i(L\pi_0(N))=0$ for every integer $i\neq 0$.
It follows that the second row in the commutative diagram
\[
\begin{tikzcd}
0\arrow[r]&
\Pic_{\overline{X}_\bullet/k}\arrow[d]\arrow[r]&
B\arrow[d]\arrow[r]&
N\arrow[r]\arrow[d,"\sim"]&
0
\\
0\arrow[r]&
\pi_0(\Pic_{\overline{X}_\bullet/k})\arrow[r]&
\pi_0(B)\arrow[r]&
\pi_0(N)\arrow[r]&
0
\end{tikzcd}
\]
is exact too.
Thus $\pi_0(B)$ is finitely generated.
Use the snake lemma to show
\[
\ker(B\rightarrow \pi_0(B))\simeq \ker(\Pic_{\overline{X}_\bullet/k}\rightarrow \pi_0(\Pic_{\overline{X}_\bullet/k})),
\]
which is a semi-abelian variety by Proposition \ref{1.38}.
\end{proof}

\begin{prop}\label{1.42}
For every $X\in Sm/k$, there is an isomorphism
\[
M_1^*(X)[-1]
\simeq
[N\rightarrow G]
\]
for some lattice $N$ and semi-abelian variety $G$.
\end{prop}
\begin{proof}
Proposition \ref{1.10} gives an isomorphism
\[
\underline{\rm Hom}_\DMeffk(M_{\geq 1}(X),\bZ(1)[2]))[-1]
\simeq
[L\rightarrow B]
\]
for some lattice $L$ and homotopy invariant Nisnevich sheaf with transfers $B$ such that $\ker(B\to \pi_0(B))$ is a semi-abelian variety.
Moreover, Proposition \ref{1.20} gives an isomorphism
\[
h_0(\underline{\rm Hom}_\DMeffk(M_{\geq 1}(X),\bZ(1)[2]))\simeq {\rm Pic}_{X/k}.
\]
It follows that there is an exact sequence
\[
L\rightarrow B\rightarrow {\rm Pic}_{X/k}\rightarrow 0.
\]
The functor $\pi_0$ is right exact since $\pi_0$ is left adjoint to the inclusion functor. Thus by Proposition \ref{1.21}, there is an exact sequence
\begin{equation}
\label{1.17.1}
L\rightarrow \pi_0(B)\rightarrow {\rm NS}(X)\rightarrow 0.
\end{equation}
By \eqref{1.15.1}, in $\DMeffk$ $M_1^*(X)[-1]$ isomorphic to the total complex of the double complex
\[
\begin{tikzcd}
L\arrow[d]\arrow[r]&B\arrow[d]\\
0\arrow[r]&{\rm NS}(X),
\end{tikzcd}\]
which is quasi-isomorphic to the complex
\begin{equation}
\label{1.17.2}
[L\rightarrow {\rm ker}(B\rightarrow {\rm NS}(X))].
\end{equation}
Due to the exactness of \eqref{1.17.1}, there is a commutative diagram with exact rows
\[
\begin{tikzcd}[column sep=tiny]
0\ar[r]&
\ker(L\to \pi_0(B))\ar[d]\ar[r]&
L\ar[r]\ar[d]&
\mathrm{im}(L\to \pi_0(B))\ar[d,"\sim"]\ar[r]&
0
\\
0\ar[r]&
\ker(B\to \pi_0(B))\ar[r]&
\ker(B\to \NS(X))\ar[r]&
\ker(\pi_0(B)\to \NS(X))\ar[r]&
0.
\end{tikzcd}
\]
This shows that \eqref{1.17.2} is quasi-isomorphic to the complex
\[
[{\rm ker}(L\rightarrow \pi_0(B))\rightarrow {\rm ker}(B\rightarrow \pi_0(B))].
\]
Since $N:={\rm ker}(L\rightarrow \pi_0(B))$ is a subsheaf of $L$, $N$ is a lattice by Lemma \ref{1.40}.
Moreover, $G:={\rm ker}(B\rightarrow \pi_0(B))$ is a semi-abelian variety.
\end{proof}

\begin{thm}\label{1.17}
For every $X\in Sm/k$, there is an isomorphism
\[
M_1^*(X)[-1]
\simeq
[N\rightarrow A]
\]
for some lattice $N$ and abelian variety $A$.
\end{thm}
\begin{proof}
By the cancellation theorem \cite{MR2804268} and Proposition \ref{1.24}, we have the vanishing
\begin{equation}
\label{1.17.3}
\begin{split}
& \uHom_\DMeffk(\bG_m,\uHom_\DMeffk(M_{\geq 1}(X),\bZ(1)[2]))
\\
\simeq &
\uHom_\DMeffk(M_{\geq 1}(X),\bZ[1])
=
0.
\end{split}
\end{equation}
By Proposition \ref{1.24} again, we have the vanishing
\begin{equation}
\label{1.17.4}
\uHom_\DMeffk(\bG_m,\NS(X))=0.
\end{equation}
Apply \eqref{1.17.3} and \eqref{1.17.4} to \eqref{1.15.1} to deduce the vanishing
\begin{equation}
\label{1.17.5}
\uHom_\DMeffk(\bG_m,M_1^*(X))=0.
\end{equation}
Due to Proposition \ref{1.42}, there is a distinguished triangle
\begin{equation}
\label{1.17.6}
G\to M_1^*(X)\to N[1]\to G[1]
\end{equation}
for some lattice $N$ and semi-abelian variety $G$.
By Proposition \ref{1.24} again, we have the vanishing
\begin{equation}
\label{1.17.7}
\uHom_\DMeffk(\bG_m,N)=0.
\end{equation}
Apply \eqref{1.17.5} and \eqref{1.17.7} to \eqref{1.17.6} to deduce the vanishing
\begin{equation}
\label{1.17.8}
\uHom_\DMeffk(\bG_m,G)=0.
\end{equation}
There is an exact sequence
\begin{equation}
0\to L\otimes \bG_m\to G\to A\to 0
\end{equation}
for some lattice $L$ and abelian variety $A$.
Together with \eqref{2.13.1} and \eqref{1.17.8}, we have the vanishing
\[
\uHom_\DMeffk(\bG_m,L\otimes \bG_m)=0.
\]
By the cancellation theorem \cite{MR2804268}, we deduce $L=0$.
Hence $G$ is isomorphic to $A$.
\end{proof}

\section{Homomorphisms and extensions of semi-abelian varieties}

In this section, we collect several results about ${\rm Hom}_\DMeffk(A,B[i])$, where $A$ and $B$ are semi-abelian varieties, and $i$ is an integer.

\begin{prop}
\label{2.18}
Let $G$ be a semi-abelian variety.
Then $G$ is a homotopy invariant Nisnevich sheaf with transfers.
\end{prop}
\begin{proof}
By \cite[Lemme 3.2.1]{MR2102056}, $G$ has a transfer structure.
Hence it remains to show that $G$ is homotopy invariant. 
By \cite[Lemme 3.3.1]{MR2102056}, this is proven when $G$ is an abelian variety or $\bG_m$.
Then use the five lemma.
\end{proof}
\begin{prop}\label{2.3}
Let $A$ and $B$ be semi-abelian varieties.
Then there is a canonical isomorphism
\[
\Hom_\DMeffk(A,B)
\simeq
\Hom_{sAV/k}(A,B),
\]
where $sAV/k$ denotes the category of semi-abelian varieties whose morphisms are homomorphisms of semi-abelian varieties.
\end{prop}
\begin{proof}
By Proposition \ref{2.18}, there is an isomorphism
\[
{\rm Hom}_\DMeffk(A,B)\simeq {\rm Hom}_{{\rm Sh}^{tr}(k)}(A,B).
\]
Let $f\colon A\rightarrow B$ a morphism of Nisnevich sheaves with transfers.
We only need to show that $f$ is a homomorphism of semi-abelian varieties.
Since $f$ is a morphism of sheaves of abelian groups, the diagram of sets 
\[\begin{tikzcd}
A(T)\times A(T)\arrow[d]\arrow[r]&A(T)\arrow[d]\\
B(T)\times B(T)\arrow[r]&B(T)
\end{tikzcd}\]
commutes for every $T\in Sm/k$, where the horizontal arrows are the multiplication maps. This means that the diagram of schemes
\[\begin{tikzcd}
A\times A\arrow[d]\arrow[r]&A\arrow[d]\\
B\times B\arrow[r]&B
\end{tikzcd}\]
commutes, where the horizontal arrows are the multiplication morphisms. Thus $f$ preserves the multiplication structure. Similarly, $f$ preserves the identity. Thus $f$ is a homomorphism of semi-abelian varieties.
\end{proof}
\begin{prop}
\label{2.2}
Let $A$ be an abelian variety.
Then there is an isomorphism
\[{\rm Ext}_{sAV/k}^1(A,\bG_m)
\simeq
{\rm Hom}_\DMeffk(A,\bZ(1)[2]).\]
\end{prop}
\begin{proof}
An element of ${\rm Ext}_{sAV/k}^1(A,\bG_m)$ is given by an exact sequence of semi-abelian varieties
\begin{equation}\label{2.2.1}
0\rightarrow \bG_m\rightarrow G\stackrel{p}\rightarrow A\rightarrow 0
\end{equation}
modulo an equivalence relation.
Note that $G$ is a Nisnevich sheaf with transfers too.

Since $G$ is a $\bG_m$-torsor on $A$, it comes from an element of $H_{fppf}^1(A,\bG_m)$. By Hilbert's Theorem 90, there is an isomorphism
\[
H_{fppf}^1(A,\bG_m)\simeq {\rm Pic}(A).
\]
From the description of this isomorphism, we deduce that there is a line bundle $V$ over $A$ such that $G\simeq V-V_0$, where $V_0$ is the zero section of $V$.
Since $V$ is Zariski locally a trivial line bundle, $G$ is Zariski locally a trivial $\bG_m$-torsor.
This means that $p$ is surjective in the category of Zariski sheaves. In particular, the sequence \eqref{2.2.1} is exact in the category of Nisnevich sheaves with transfers. Thus we can view ${\rm Ext}_{sAV/k}^1(A,\bG_m)$ as a subset of ${\rm Ext}_{{\rm Sh}^{tr}(k)}^1(A,\bG_m)$.

Suppose that
\[
0\rightarrow \bG_m\rightarrow F\rightarrow A\rightarrow 0
\]
is an exact sequence of Nisnevich sheaves with transfers. Then $F$ is representable by \cite[Proposition 17.4]{MR0213365}, so $F$ is isomorphic to a semi-abelian variety.
This establishes an isomorphism
\[
{\rm Ext}_{sAV/k}^1(A,\bG_m)
\simeq
{\rm Ext}_{{\rm Sh}^{tr}(k)}^1(A,\bG_m).
\]

Since $\bG_m$ is ${\bf A}^1$-local, there are isomorphisms
  \[\begin{split}&{\rm Hom}_\DMeffk(A,\bZ(1)[2])\simeq {\rm Hom}_\DMeffk(A,\bG_m[1])\\
  \simeq& {\rm Hom}_{{\rm D}({\rm Sh}^{tr}(k))}(A,\bG_m[1])\simeq  {\rm Ext}_{{\rm Sh}^{tr}(k)}^1(A,\bG_m).\end{split}\]
Combine the above equations to conclude.
\end{proof}

\begin{prop}\label{2.4}
Let $A$ be an abelian variety.
Then for every integer $i>0$, we have the vanishing
\[
{\rm Hom}_\DMeffk(A[i],\bZ(1)[2])=0.
\]
\end{prop}
\begin{proof}
If $i>1$, then $A[i-2]$ is $t$-positive for the $0$-motivic structure, and $\bZ(1)[1]\simeq \bG_m$ is $t$-negative for the $0$-motivic structure. Thus we are done in this case.

If $i=1$, then there is an isomorphism
\[
{\rm Hom}_\DMeffk(A[1],\bZ(1)[2])\simeq{\rm Hom}_{sAV/k}(A,\bG_m)
\]
by Proposition \ref{2.3}.
There are isomorphisms
\[
{\rm Hom}_{Sch/k}(A,\bG_m)\simeq H_{Nis}^0(A,\bG_m)\simeq k^*.
\]
Thus every morphism $f\colon A\rightarrow \bG_m$ of schemes is a constant morphism.
If $f$ is a homomorphism of semi-abelian varieties, $f$ should be the zero morphism.
This shows the vanishing ${\rm Hom}_{sAV/k}(A,\bG_m)=0$.
\end{proof}
\begin{lemma}\label{2.23}
Let $A$ be an abelian variety. Then
\[
\underline{\rm Hom}_{\DMeffk}(A,\bZ(1)[2])
\]
is $t$-negative for the $0$-motivic $t$-structure.
\end{lemma}
\begin{proof}
We need to show the vanishing
\[
{\rm Hom}_{\DMeffk}(M(X)[i],\underline{\rm Hom}_{\DMeffk}(A,\bZ(1)[2]))=0
\]
for every integral $X\in Sm/k$ and integer $i>0$.
By using the distinguished triangle
\[
M_{\geq 1}(X)\rightarrow M(X)\rightarrow \bZ\rightarrow M_{\geq 1}(X)[1]
\]
we only need to show the vanishings
\begin{equation}
\label{1.3.2}
{\rm Hom}_{\DMeffk}(\bZ[i],\underline{\rm Hom}_{\DMeffk}(A,\bZ(1)[2]))=0
\end{equation}
and
\begin{equation}
\label{1.3.1}
{\rm Hom}_{\DMeffk}(M_{\geq 1}(X)[i],\underline{\rm Hom}_{\DMeffk}(A,\bZ(1)[2]))=0
\end{equation}
for every integer $i>0$.
Proposition \ref{2.4} gives \eqref{1.3.2}, so it remains to show \eqref{1.3.1}.

By adjunction, there is an isomorphism
\[
\begin{split}
&{\rm Hom}_{\DMeffk}(M_{\geq 1}(X)[i],\underline{\rm Hom}_{\DMeffk}(A,\bZ(1)[2]))\\
\simeq& {\rm Hom}_{\DMeffk}(A[i],\underline{\rm Hom}_{\DMeffk}(M_{\geq 1}(X),\bZ(1)[2])).\end{split}
\]
There is a distinguished triangle
\[
M_1^*(X)\rightarrow \underline{\rm Hom}_{\DMeffk}(M_{\geq 1}(X),\bZ(1)[2])\rightarrow {\rm NS}(X)\rightarrow M_1^*(X)[1].
\]
Since ${\rm NS}(X)$ is constant, we have the vanishing
\[
{\rm Hom}_\DMeffk(A[i],{\rm NS}(X))=0
\]
for every integer $i> 0$ by Proposition \ref{1.33}.
Hence it suffices to show the vanishing
\[{\rm Hom}_\DMeffk(A[i],M_1^*(X))=0\]
for every integer $i>0$.

Due to Theorem \ref{1.17}, there exists a distinguished triangle
\[
L\to B\to M_1^*(X)\to L[1]
\]
for some lattice $L$ and semi-abelian variety $B$.
Hence it suffices to show the vanishing
\[{\rm Hom}_\DMeffk(A[i],\bZ[1])=0\text{ and }{\rm Hom}_\DMeffk(A[i],B)=0\]
for every integer $i>0$.
By Proposition \ref{1.33}, the first one holds.
The second one holds since $A$ and $B$ are in the heart of the $0$-motivic $t$-structure.
\end{proof}

\section{Motivic duals of abelian varieties}
The purposes of this section is to show that there is an isomorphism
\[
\tau_{\geq 0}({\underline{\rm Hom}}_{\DMeffk}(A,\bZ(1)[2]))\simeq A^\vee
\]
for every abelian variety $A$, where $A^\vee$ denotes the dual abelian variety of $A$.
By Lemma \ref{2.23}, this is equivalent to showing that there is an isomorphism
\[
h_0(\underline{\rm Hom}_\DMeffk(A,\bZ(1)[2]))
\xrightarrow{\sim}
A^\vee.
\]

\begin{none}
\label{2.19}
For every $X\in Sm/k$, there is a semi-abelian group scheme $\Alb(X)$ together with a morphism
\[
\Alb\colon X\rightarrow \Alb(X)
\]
that is universal among all morphisms from $X$ to semi-abelian group schemes, see \cite{SCC_1958-1959__4__A11_0} and \cite{RamachandranDuality}.
The semi-abelian group scheme $\Alb(X)$ is called the \emph{Albanese scheme} of $X$, and the morphism $\Alb$ is called the \emph{Albanese morphism}.
There is an induced morphism
\[
\bZ^{tr}(X)\rightarrow \Alb(X)
\]
of Nisnevich sheaves with transfers.
Due to \cite[Eq.\ 4]{RamachandranDuality}, there is a canonical isomorphism
\begin{equation}
\label{2.19.1}
\pi_0(\Alb(X))\simeq \bZ^{\pi_0(X)}.
\end{equation}

Let $\Alb^0(X)$ be the connected component of the identity of $\Alb(X)$.
Thus $\Alb^0(X)$ is a semi-abelian variety, which is called the \emph{Albanese variety} of $X$.
There is an exact sequence
\[
0\to \Alb^0(X)\to \Alb(X)\to \bZ^{\pi_0(X)}\to 0.
\]
There is also an induced morphism
\[
\Alb\colon M_{\geq 1}(X)\rightarrow \Alb^0(X),
\]
which is called the \emph{Albanese morphism} again.
\end{none}
\begin{none}\label{1.4}
Let $A$ be an abelian variety.
We will construct a commutative diagram
\begin{equation}
\label{1.4.1}
\begin{tikzcd}
&
\underline{\rm Hom}_{\DMeffk}(A,\bZ(1)[2])\ar[ld,dashrightarrow,"\eta"']\ar[rd,"0"]\ar[d,"q"]
\\
A^\vee\ar[r]&
\underline{\rm Hom}_{\DMeffk}(M_{\geq 1}(A),\bZ(1)[2])\ar[r,"p"]&
{\rm NS}(A)\ar[r]&
 A^{\vee}[1]
\end{tikzcd}
\end{equation}
as follows.
Apply Proposition \ref{1.25} to $A$ to obtain the lower row of \eqref{1.4.1}, which is a distinguished triangle.
The Albanese morphism $M_{\geq 1}(A)\to A$ naturally induces $q$.

Let us show $pq=0$.
It suffices to show that the homomorphism
\begin{equation}
\label{1.4.2}
{\rm Hom}_{\DMeffk}(M(X)[i],\underline{\rm Hom}_{\DMeffk}(A,\bZ(1)[2]))\rightarrow {\rm Hom}_{\DMeffk}(M(X)[i],{\rm NS}(A))
\end{equation}
is $0$ for every $X\in Sm/k$ and integer $i$.
We only need to consider the case when $X={\rm Spec}\,k$ and $i=0$ since ${\rm NS}(A)$ is a constant sheaf.
In this case, owing to Proposition \ref{2.2} the homomorphism \eqref{1.4.2} becomes a homomorphism
\[
{\rm Ext}_{sAV/k}^1(A,\bG_m)
\to
{\rm NS}(A).
\]
This is zero due to \cite[Proposition 17.6]{MR0213365}.
Thus in \eqref{1.4.1}, we can construct $\eta$.
\end{none}

\begin{thm}\label{2.21}
Let $A$ be an abelian variety.
Then the morphism $\eta$ in \eqref{1.4.1} naturally induces an isomorphism
\begin{equation}
\label{2.21.1}
\varphi
\colon
h_0(\underline{\rm Hom}_\DMeffk(A,\bZ(1)[2]))
\xrightarrow{\sim}
A^\vee.
\end{equation}
\end{thm}
\begin{proof}
The left one in \eqref{2.21.1} is the sheaf associated with the presheaf
\[
X\in Sm/k \mapsto {\rm Hom}_\DMeffk(M(X),\underline{\rm Hom}_{\DMeffk}(A,\bZ(1)[2])).
\]
Hence it suffices to show that there is a canonical isomorphism
\[
{\rm Hom}_\DMeffk(M(X),\underline{\rm Hom}_{\DMeffk}(A,\bZ(1)[2]))
\xrightarrow{\sim}
{\rm Hom}_\DMeffk(M(X),A^\vee)
\]
for every $X\in Sm/k$.

Since $M(X)\simeq \bZ^{\pi_0(X)}\oplus M_{\geq 1}(X)$, we reduce to Propositions \ref{1.5} and \ref{2.20} below.
\end{proof}

\begin{prop}\label{1.5}
Let $A$ be an abelian variety.
For every $X\in Sm/k$, the morphism $\eta$ in \eqref{1.4.1} naturally induces an isomorphism
\[
{\rm Hom}_\DMeffk(M_{\geq 1}(X),\underline{\rm Hom}_{\DMeffk}(A,\bZ(1)[2]))\rightarrow {\rm Hom}_\DMeffk(M_{\geq 1}(X),A^\vee).
\]
\end{prop}
\begin{proof}
By Proposition \ref{1.24} for every integer $i$, we have the vanishing
\[
{\rm Hom}_{\DMeffk}(M_{\geq 1}(X)[i],{\rm NS}(A))=0.
\]
Since the lower row of \eqref{1.4.1} is a distinguished triangle we have an isomorphism
\[
\begin{split}
&{\rm Hom}_\DMeffk(M_{\geq 1}(X),\underline{\rm Hom}_\DMeffk (M_{\geq 1}(A),\bZ(1)[2]))\\
\simeq& {\rm Hom}_\DMeffk(M_{\geq 1}(X),A^\vee).
\end{split}
\]
Then it suffices to show that the Albanese morphism $M_{\geq 1}(A)\to A$ naturally induces an isomorphism
\[
\begin{split}
&{\rm Hom}_{\DMeffk}(M_{\geq 1}(X),\underline{\rm Hom}_{\DMeffk}(A,\bZ(1)[2]))
\\
\rightarrow &{\rm Hom}_{\DMeffk}(M_{\geq 1}(X),\underline{\rm Hom}_{\DMeffk}(M_{\geq 1}(A),\bZ(1)[2])).
\end{split}
\]
By adjunction, this becomes a homomorphism
\[
\begin{split}
&{\rm Hom}_{\DMeffk}(A,\underline{\rm Hom}_{\DMeffk}(M_{\geq 1}(X),\bZ(1)[2]))
\\
\rightarrow &{\rm Hom}_{\DMeffk}(M_{\geq 1}(A),\underline{\rm Hom}_{\DMeffk}(M_{\geq 1}(X),\bZ(1)[2])).
\end{split}
\]
Due to Propositions \ref{1.24} and \ref{1.33} for every integer $i\leq 1$, we have the vanishings
\[{
\rm Hom}_{\DMeffk}(M_{\geq 1}(A),{\rm NS}(X)[i])=0 \text{ and }{\rm Hom}_{\DMeffk}(A,{\rm NS}(X)[i])=0.
\]
Hence using the distinguished triangle \eqref{1.15.1}, it suffices to show that the Albanese morphism $M_{\geq 1}(A)\to A$ naturally induces an isomorphism
\[
{\rm Hom}_{\DMeffk}(A,M_1^*(X))\rightarrow {\rm Hom}_{\DMeffk}(M_{\geq 1}(A),M_1^*(X)).
\]
By Theorem \ref{1.17}, there is a distinguished triangle
\[
B\rightarrow M_1^*(X)\rightarrow L[1]\rightarrow B[1]
\]
in $\DMeffk$, where $B$ is an abelian variety, and $L$ is a lattice.
By Propositions \ref{1.24} and \ref{1.33} for every integer $i\leq 1$, we again have the vanishings
\[
{\rm Hom}_{\DMeffk}(M_{\geq 1}(A),L[i])=0
\text{ and }
{\rm Hom}_{\DMeffk}(A,L[i])=0.
\]
Hence it remains to show that the Albanese morphism $M_{\geq 1}(A)\to A$ naturally induces an isomorphism
\[
{\rm Hom}_{\DMeffk}(A,B)\rightarrow {\rm Hom}_{\DMeffk}(M_{\geq 1}(A),B).
\]
This follows from Proposition \ref{2.3} and the universality of the Albanese morphism $M_{\geq 1}(A)\rightarrow A$.
\end{proof}
\begin{prop}\label{2.20}
Let $A$ be an abelian variety.
Then the morphism $\eta$ in \eqref{1.4.1} naturally induces an isomorphism
\[
{\rm Hom}_\DMeffk(A,\bZ(1)[2])\rightarrow {\rm Hom}_\DMeffk(\bZ,A^\vee).
\]
\end{prop}
\begin{proof}
This is equivalent to showing that the naturally induced homomorphism
\[
{\rm Hom}_{\DMeffk}(A,\bZ(1)[2])\rightarrow {\rm Hom}_{\DMeffk}(M_{\geq 1}(A),\bZ(1)[2])\simeq {\rm Pic}(A)
\]
is injective and has image ${\rm Pic}^0(A)$.
Here, the right isomorphism comes from Proposition \ref{1.25}.
This follows from Proposition \ref{2.2} and \cite[Proposition 17.6]{MR0213365}.
\end{proof}

\begin{prop}\label{2.22}
The morphism $\varphi$ in \eqref{2.21.1} is functorial in $A$.
More precisely, for every homomorphism $f\colon A'\to A$ of abelian varieties there is a commutative diagram
  \[\begin{tikzcd}
    h_0(\underline{\rm Hom}_\DMeffk(A,\bZ(1)[2]))\arrow[d,"{h_0(\uHom_\DMeffk(f,\bZ(1)[2]))}"']\arrow[r,"\varphi"]&A^\vee\arrow[d,"f^\vee"]\\
    h_0(\underline{\rm Hom}_\DMeffk(A',\bZ(1)[2]))\arrow[r,"\varphi"]&A'^\vee.
  \end{tikzcd}
\]
\end{prop}
\begin{proof}
There is a commutative diagram
  \begin{equation}\label{2.22.1}\begin{tikzcd}
    M_{\geq 1}(A')\arrow[r,"u"]\arrow[d,"g"']&A'\arrow[d,"f"]\\
    M_{\geq 1}(A)\arrow[r,"u'"]&A
  \end{tikzcd}\end{equation}
  in $\DMeffk$, where $g:=M_{\geq 1}(f)$, and $u$ and $u'$ are the Albanese morphisms.
Lemma \ref{2.23} gives morphisms
\begin{gather*}
h_0(\underline{\rm Hom}_\DMeffk(A,\bZ(1)[2]))\stackrel{v}\rightarrow \underline{\rm Hom}_\DMeffk(A,\bZ(1)[2]),
\\
h_0(\underline{\rm Hom}_\DMeffk(A',\bZ(1)[2]))\stackrel{v'}\rightarrow\underline{\rm Hom}_\DMeffk(A',\bZ(1)[2])
\end{gather*}
in $\DMeffk$.
There is an induced diagram
\[
\begin{tikzcd}[column sep=small, row sep=small]
&&A^\vee\arrow[rd,"i"]\arrow[dd,"f^\vee",near start]
\\
h_0(\underline{\rm Hom}_\DMeffk(A,\bZ(1)[2]))\arrow[r,"v"]\arrow[dd,"h_0(f')"']&\underline{\rm Hom}_\DMeffk(A,\bZ(1)[2])\arrow[dd,"f'"]\arrow[rr,crossing over,"i\eta",near start]\arrow[ru,"\eta"]&&{\rm Pic}_{A/k}\arrow[dd,"g"]
\\
&&A'^\vee\arrow[rd,"i'"]
\\
h_0(\underline{\rm Hom}_\DMeffk(A',\bZ(1)[2]))\arrow[r,"v'"]&\underline{\rm Hom}_\DMeffk(A',\bZ(1)[2])\arrow[ru,"\eta'"]\arrow[rr,"i'\eta'",near start]&&{\rm Pic}_{A'/k}
\end{tikzcd}
\]
in $\DMeffk$, where $i$, $i'$, $\eta$, and $\eta'$ come from Proposition \ref{1.25} and \eqref{1.4.1}. By taking $\underline{\rm Hom}_\DMeffk(-,\bZ(1)[2])$ to \eqref{2.22.1}, we see that the right front square commutes.
We have relations
\[
i'f^\vee \eta v=gi\eta v=i'\eta'f'v=i'\eta'v'h_0(f').
\]
There is a distinguished triangle
\[
A'^\vee \rightarrow {\rm Pic}_{A'/k}\rightarrow {\rm NS}(A')\rightarrow A'^\vee[1]
\]
in $\DMeffk$.
Hence to show $f^\vee \eta v=\eta'v'h_0(f)$, it suffices to show the vanishing
\[
{\rm Hom}_\DMeffk(h_0(\underline{\rm Hom}_\DMeffk(A^\vee,\bZ(1)[2])),{\rm NS}(A')[i])=0
\]
for $i=-1,0$.
By Theorem \ref{2.21}, it suffices to show the vanishing
\[
{\rm Hom}_\DMeffk(A,{\rm NS}(A')[i])=0
\]
for $i=-1,0$. This follows from Proposition \ref{1.33}.
\end{proof}

\section{Motivic duals of semi-abelian varieties}

\begin{none}\label{2.24}
Let $G$ be a semi-abelian variety with an exact sequence
\begin{equation}\label{2.24.1}
0\rightarrow L\otimes \bG_m\rightarrow G\rightarrow A\rightarrow 0
\end{equation}
of group schemes, where $L$ is a lattice, and $A$ is an abelian variety.
Then the Cartier dual of $G$ can be written as $G^\vee:=[L^\vee\rightarrow A^\vee][1]$, where $L^\vee$ denotes the dual lattice of $L$, and $A^\vee$ denotes the dual abelian variety of $A$.
We refer to \cite[Paragraph 10.2.11]{zbMATH03375638} for the definition of the Cartier dual.

There is a commutative diagram
\begin{equation}
\label{2.24.2}
\begin{tikzcd}[row sep=small]
\underline{\rm Hom}_\DMeffk(L^\vee[1],\bZ(1)[2])\arrow[r,"\sim"]\arrow[d]&L\otimes \bG_m\arrow[d]\\
\underline{\rm Hom}_\DMeffk(G^\vee,\bZ(1)[2])\arrow[d]&G\arrow[d]\\
\underline{\rm Hom}_\DMeffk(A^\vee, \bZ(1)[2])\arrow[r,"\eta"]\arrow[d]&A\arrow[d]\\
\underline{\rm Hom}_\DMeffk(L^\vee,\bZ(1)[2])\arrow[r,"\sim"]&L\otimes \bG_m[1]
\end{tikzcd}
\end{equation}
in $\DMeffk$ whose columns are distinguished triangles.
This diagram gives a morphism
\[
\mu\colon \underline{\rm Hom}_\DMeffk(G^\vee,\bZ(1)[2])\rightarrow G
\]
in $\DMeffk$ such that the above diagram still commutes after adding this.
\end{none}

\begin{lemma}\label{2.28}
Let $G$ be a semi-abelian variety. Then $\underline{\rm Hom}_{\DMeffk}(G^\vee,\bZ(1)[2])$ is $t$-negative for the $0$-motivic $t$-structure.
\end{lemma}
\begin{proof}
Since $\underline{\rm Hom}_{\DMeffk}(L^\vee[1],\bZ(1)[2])\simeq L\otimes \bG_m$ is in the heart of the $0$-motivic $t$-structure, we are done by Lemma \ref{2.23}.
\end{proof}

\begin{prop}\label{2.25}
Let $G$ be a semi-abelian variety.
Then there is an isomorphism
\begin{equation}
\label{2.25.2}
\tau
\colon
h_0(\underline{\rm Hom}_\DMeffk(G^\vee,\bZ(1)[2]))\rightarrow G.
\end{equation}
\end{prop}
\begin{proof}
Suppose we have an exact sequence of the form \eqref{2.24.1}.
From \eqref{2.24.2} we have a commutative diagram
\begin{equation}
\label{2.25.1}
\begin{tikzcd}[row sep=small]
  &0\arrow[d]\\
  h_0(\underline{\rm Hom}_\DMeffk(L^\vee[1],\bZ(1)[2]))\arrow[r,"\sim"]\arrow[d,"r"]&L\otimes \bG_m\arrow[d]\\
   h_0( \underline{\rm Hom}_\DMeffk(G^\vee,\bZ(1)[2]))\arrow[d]\arrow[r,"\tau"]&G\arrow[d]\\
   h_0( \underline{\rm Hom}_\DMeffk(A^\vee, \bZ(1)[2]))\arrow[r,"\varphi"]\arrow[d]&A\arrow[d]\\
    0&0
\end{tikzcd}
\end{equation}
of Nisnevich sheaves with transfers.
By Theorem \ref{2.21}, $\varphi$ is an isomorphism.
Since $\tau \circ r$ is injective, $r$ is injective too.
Hence $\tau$ is an isomorphism by the five lemma.
\end{proof}

\begin{prop}\label{2.26}
The morphism $\varphi$ in \eqref{2.25.2} is functorial in $G$.
More precisely, for every homomorphism $f\colon G\to G'$ of semi-abelian varieties there is a commutative diagram
\[\begin{tikzcd}
h_0(\underline{\rm Hom}_\DMeffk(G^\vee,\bZ(1)[2]))\arrow[d,"f'"']\arrow[r,"\tau"]&G\arrow[d,"f"]\\
h_0(\underline{\rm Hom}_\DMeffk(G'^\vee,\bZ(1)[2]))\arrow[r,"\tau'"]&G',
\end{tikzcd}\]
where $f':=h_0(\uHom_\DMeffk(f^\vee,\bZ(1)[2]))$, and $\tau$ and $\tau'$ are obtained by Proposition \ref{2.25}.
\end{prop}
\begin{proof}
There is a commutative diagram
\[\begin{tikzcd}
0\arrow[r]&L\otimes \bG_m\arrow[d,"h"']\arrow[r]&G\arrow[d,"f"]\arrow[r]&A\arrow[d,"g"]\arrow[r]&0\\
0\arrow[r]&L'\otimes \bG_m\arrow[r]&G'\arrow[r]&A'\arrow[r]&0
\end{tikzcd}\]
of group schemes with exact rows, where $L$ and $L'$ are lattices, and $A$ and $A'$ are abelian varieties.
Proposition \ref{2.22} tells that there is a commutative diagram
\[
\begin{tikzcd}
h_0(\underline{\rm Hom}_\DMeffk(A^\vee,\bZ(1)[2]))\arrow[d,"g'"']\arrow[r,"\varphi"]&A\arrow[d,"g"]\\
h_0(\underline{\rm Hom}_\DMeffk(A'^\vee,\bZ(1)[2]))\arrow[r,"\varphi'"]&A'
\end{tikzcd}
\]
of Nisnevich sheaves with transfers.
  
There is a commutative diagram
  \[\begin{tikzcd}[row sep=tiny]
    0\arrow[d]&0\arrow[d]\\
    h_0(\uHom_\DMeffk(L^\vee[1],\bZ(1)[2]))\arrow[d]\arrow[r]&L'\otimes \bG_m\arrow[d]\\
    h_0(\underline{\rm Hom}_\DMeffk(G^\vee,\bZ(1)[2]))\arrow[d]&G'\arrow[d]\\
    h_0(\underline{\rm Hom}_\DMeffk(A^\vee,\bZ(1)[2])) \arrow[r,"g\varphi"]\arrow[d]&A'\arrow[d]\\
    0&0
  \end{tikzcd}\]
of Nisnevich sheaves with transfers.
The right column is exact.
As observed in the proof of Proposition \ref{2.25} the left column is exact too.
The above diagram still commutes after adding $f\tau \colon h_0(\underline{\rm Hom}_\DMeffk(G^\vee,\bZ(1)[2])\rightarrow G'$.
Since $g\varphi=\varphi'g'$, the same holds if we add $\tau'f'$.
Hence there is a commutative diagram
\[
\begin{tikzcd}[row sep=tiny]
0\arrow[d]&0\arrow[d]\\
h_0(\uHom_\DMeffk(L^\vee[1],\bZ(1)[2]))\arrow[d]\arrow[r,"0"]&L'\otimes \bG_m\arrow[d]\\
h_0(\underline{\rm Hom}_\DMeffk(G^\vee,\bZ(1)[2]))\arrow[d]\ar[r,"f\tau-\tau'f'"]&G'\arrow[d]\\
h_0(\underline{\rm Hom}_\DMeffk(A^\vee,\bZ(1)[2])) \arrow[r,"0"]\arrow[d]&A'\arrow[d]\\
0&0.
\end{tikzcd}
\]
This means that $f\tau-\tau f'$ factors through a morphism
\[
u\colon h_0(\underline{\rm Hom}_\DMeffk(A^\vee,\bZ(1)[2])) \to L'\otimes \bG_m.
\]
From Proposition \ref{2.4} and Theorem \ref{2.21} we deduce $u=0$, which implies $f\tau-\tau f'=0$.
\end{proof}
\begin{prop}\label{2.31}
Let $G$ be a semi-abelian variety.
For every $X\in Sm/k$, there is an isomorphism
\[
{\rm Hom}_\DMeffk(M_{\geq 1}(X),\underline{\rm Hom}_{\DMeffk}(G^\vee,\bZ(1)[2]))\rightarrow {\rm Hom}_\DMeffk(M_{\geq 1}(X),G).
\]
\end{prop}
\begin{proof}
There is an exact sequence
\[0\rightarrow L\otimes \bG_m\rightarrow G\rightarrow A\rightarrow 0\]
of group schemes, where $L$ is a lattice, and $A$ is an abelian variety.
From \eqref{2.24.2} we have a commutative diagram
  \[\begin{tikzpicture}[baseline= (a).base]
    \node[scale=.9] (a) at (0,0)
    {\begin{tikzcd}[column sep=tiny, row sep=small]
    {\rm Hom}_\DMeffk(M_{\geq 1}(X),\underline{\rm Hom}_\DMeffk(A^\vee,\bZ(1)[1]))\arrow[r]\arrow[d]&{\rm Hom}_\DMeffk(M_{\geq 1}(X),A[-1])\arrow[d]\\
    {\rm Hom}_\DMeffk(M_{\geq 1}(X),\underline{\rm Hom}_\DMeffk(L^\vee,\bZ(1)[1]))\arrow[r,"\sim"]\arrow[d]&{\rm Hom}_\DMeffk(M_{\geq 1}(X),L\otimes \bG_m[-1])\arrow[d]\\
    {\rm Hom}_\DMeffk(M_{\geq 1}(X),\underline{\rm Hom}_\DMeffk(G^\vee,\bZ(1)[2]))\arrow[r]\arrow[d]&{\rm Hom}_\DMeffk(M_{\geq 1}(X),G)\arrow[d]\\
    {\rm Hom}_\DMeffk(M_{\geq 1}(X),\underline{\rm Hom}_\DMeffk(A^\vee,\bZ(1)[2]))\arrow[d]\arrow[r]&{\rm Hom}_\DMeffk(M_{\geq 1}(X),A)\arrow[d]\\
    {\rm Hom}_\DMeffk(M_{\geq 1}(X),\underline{\rm Hom}_\DMeffk(L^\vee,\bZ(1)[2]))\arrow[r,"\sim"]&{\rm Hom}_\DMeffk(M_{\geq 1}(X),L\otimes \bG_m).
  \end{tikzcd}};
  \end{tikzpicture}\]
By Proposition \ref{1.5}, the fourth horizontal arrow is an isomorphism. Hence by the five lemma, to show that the third horizontal arrow is an isomorphism, it suffices to show that the first horizontal arrow is an isomorphism.
  
Owing to Lemma \ref{2.23}, $\underline{\rm Hom}_\DMeffk(A^\vee,\bZ(1)[2])$ is $t$-negative for the $0$-motivic $t$-structure. Since $M_{\geq 1}(X)$ is $t$-positive for the $0$-motivic $t$-structure, we have the vanishing
 \[ {\rm Hom}_\DMeffk(M_{\geq 1}(X),\underline{\rm Hom}_\DMeffk(A^\vee,\bZ(1)[2])[-1])=0.\]
We also have the vanishing
\[
{\rm Hom}_\DMeffk(M_{\geq 1}(X),A[-1])=0
\]
since $A$ is in the heart of the $0$-motivic $t$-structure.
Thus the first horizontal arrow is an isomorphism.
\end{proof}

\begin{thm}
\label{2.29}
Suppose $G$ is a semi-abelian variety and $X\in Sm/k$.
Then there is an isomorphism
\[
{\rm Hom}_\DMeffk(M_{\geq 1}(X),G)\stackrel{\sim}\rightarrow {\rm Hom}_\DMeffk(G^\vee,M_1^*(X))
\]
that is functorial in $G$.
\end{thm}
\begin{proof}
Let $f\colon G\rightarrow G'$ be a homomorphism of semi-abelian varieties.
From \eqref{1.15.1} and Proposition \ref{2.26} we can make a commutative diagram
\[\begin{tikzpicture}[baseline= (a).base]
\node[scale=.73] (a) at (0,0)
{\begin{tikzcd}[column sep=tiny, row sep=small]
    {\rm Hom}_\DMeffk(M_{\geq 1}(X),G)\arrow[r]\arrow[d,"u"',leftarrow]&{\rm Hom}_\DMeffk(M_{\geq 1}(X),G')\arrow[d,"u'",leftarrow]\\
    {\rm Hom}_\DMeffk(M_{\geq 1}(X),h_0(\underline{\rm Hom}_\DMeffk(G^\vee,\bZ(1)[2])))\arrow[d,"v"']\arrow[r]&{\rm Hom}_\DMeffk(M_{\geq 1}(X),h_0(\underline{\rm Hom}_\DMeffk(G'^\vee,\bZ(1)[2])))\arrow[d,"v'"]\\
    {\rm Hom}_\DMeffk(M_{\geq 1}(X),\underline{\rm Hom}_\DMeffk(G^\vee,\bZ(1)[2]))\arrow[d,"\sim"']\arrow[r]&{\rm Hom}_\DMeffk(M_{\geq 1}(X),\underline{\rm Hom}_\DMeffk(G'^\vee,\bZ(1)[2]))\arrow[d,"\sim"]\\
    {\rm Hom}_\DMeffk(G^\vee,\underline{\rm Hom}_\DMeffk(M_{\geq 1}(X),\bZ(1)[2]))\arrow[d,leftarrow,"w"']\arrow[r]&{\rm Hom}_\DMeffk(G'^\vee,\underline{\rm Hom}_\DMeffk(M_{\geq 1}(X),\bZ(1)[2]))\arrow[d,leftarrow,"w'"]\\
    {\rm Hom}_\DMeffk(G^\vee,M_1^*(X))\arrow[r]& {\rm Hom}_\DMeffk(G^\vee,M_1^*(X)).
    \end{tikzcd}};
  \end{tikzpicture}\]
The morphisms $u$ and $u'$ are isomorphisms by Proposition \ref{2.25}.
Since $M_{\geq 1}(X)$ is $t$-positive for the $0$-motivic $t$-structure, $v$ and $v'$ are isomorphisms by Lemma \ref{2.28}.
  
It remains to show that $w$ is an isomorphism.
By \eqref{1.15.1}, it suffices to show the vanishing
\begin{equation}
\label{2.29.1}
{\rm Hom}_\DMeffk(G^\vee,{\rm NS}(X)[i])=0
\end{equation}
  for $i=-1,0$.
  Since there is a distinguished triangle
  \[L^\vee\rightarrow A^\vee\rightarrow G^\vee\rightarrow L^\vee[1]\]
  in $\DMeffk$, it suffices to show the vanishings
  \[{\rm Hom}_\DMeffk(L^\vee[1],{\rm NS}(X)[i])=0 \text{ and }{\rm Hom}_\DMeffk(A^\vee,{\rm NS}(X)[i])=0\]
  for $i=-1,0$. The first one holds since $L^\vee$ and ${\rm NS}(X)$ are in the heart of the $0$-motivic $t$-structure. Since there is a distinguished triangle
  \[\bZ^r\rightarrow \bZ^s\rightarrow {\rm NS}(X)\rightarrow \bZ^r[1]\]
  in $\DMeffk$ for some integers $r,s\geq 0$, it suffices to show the vanishing
  \[{\rm Hom}_\DMeffk(A^\vee,\bZ[i])=0\]
  for every integer $i\leq 1$. This follows from Proposition \ref{1.33}.
\end{proof}

\begin{rmk}
\label{2.32}
By forgetting $w$ and $w'$ in the proof of Theorem \ref{2.29}, we also have an isomorphism
\begin{equation}
\label{2.32.1}
{\rm Hom}_\DMeffk(M_{\geq 1}(X),G)\stackrel{\sim}\rightarrow {\rm Hom}_\DMeffk(G^\vee,\underline{\rm Hom}_\DMeffk(M_{\geq 1}(X),\bZ(1)[2]))
\end{equation}
that is functorial in $G$.
Moreover, this is trivially functorial in $X$.
\end{rmk}

\begin{thm}
\label{2.30}
Suppose $X\in Sm/k$.
Then $M_1^*(X)$ is isomorphic to the Cartier dual of $\Alb^0(X)$.
\end{thm}
\begin{proof}
By Theorem \ref{2.29}, a morphism
\[
M_{\geq 1}(X)\rightarrow G
\]
in $\DMeffk$ is universal among all morphisms from $M_{\geq 1}(X)$ to semi-abelian varieties if and only if the corresponding morphism
\[
G^\vee\rightarrow M_1^*(X)
\]
is universal among all morphisms from the Cartier duals of semi-abelian varieties to $M_1^*(X)$.
The identity morphism
\[
M_1^*(X)\rightarrow M_1^*(X)
\]
is the solution to the universal problem since $M_1^*(X)$ is the Cartier dual of a semi-abelian variety by Theorem \ref{1.17}.
This means that $M_1^*(X)$ is isomorphic to the Cartier dual of $\Alb^0(X)$.
\end{proof}

\begin{none}
\label{2.33}
Suppose $f\colon X\to Y$ is a morphism in $Sm/k$.
As a consequence of Theorem \ref{2.30}, we see that the morphism $\Alb^0(f)\colon \Alb^0(X)\to \Alb^0(Y)$ naturally induces a morphism $f^*\colon M_1^*(Y)\to M_1^*(X)$.
\end{none}

\begin{prop}
\label{2.34}
Suppose $X\in Sm/k$.
Then the morphism
\[
M_1^*(X)\to \underline{\rm Hom}(M_{\geq 1}(X),\bZ(1)[2])
\]
in \eqref{1.15.1} corresponds to the Albanese morphism $M_{\geq 1}(X)\to \Alb^0(X)$ via \eqref{2.32.1}.
\end{prop}
\begin{proof}
In the proof of Theorem \ref{2.29}, it is shown that the homomorphism
\[
w\colon \underline{\rm Hom}_\DMeffk(G^\vee,M_1^*(X))
\to
\underline{\rm Hom}_\DMeffk(G^\vee,\underline{\rm Hom}_{\DMeffk}(M_{\geq 1}(X),\bZ(1)[2]))
\]
is an isomorphism for every semi-abelian variety $G$.
This means that the morphism $M_1^*(X)\to \underline{\rm Hom}(M_{\geq 1}(X),\bZ(1)[2])$ is universal among all morphisms from the Cartier duals of semi-abelian varieties to $\underline{\rm Hom}(M_{\geq 1}(X),\bZ(1)[2])$.
From \eqref{2.32.1} and the universality of the Albanese morphism $M_{\geq 1}(X)\to \Alb^0(X)$ we conclude.
\end{proof}

\begin{prop}
\label{2.36}
Suppose $f\colon X\to Y$ is a morphism in $Sm/k$.
Then there is a commutative diagram
\begin{equation}
\label{2.36.1}
\begin{tikzcd}
M_1^*(Y)\ar[d,"f^*"']\ar[r]&
\underline{\rm Hom}_\DMeffk (M_{\geq 1}(Y),\bZ(1)[2])\ar[d,"{\uHom_{\DMeffk}(M_{\geq 1}(f),\bZ(1)[2])}"]
\\
M_1^*(X)\ar[r]&
\underline{\rm Hom}_\DMeffk (M_{\geq 1}(X),\bZ(1)[2]),
\end{tikzcd}
\end{equation}
where the horizontal arrows are given by \eqref{1.15.1}.
\end{prop}
\begin{proof}
There is a commutative diagram
\begin{equation}
\label{2.36.2}
\begin{tikzcd}
M_{\geq 1}(X)\ar[d,"M_{\geq 1}(f)"']\ar[r]&
\Alb^0(X)\ar[d,"\Alb^0(f)"]
\\
M_{\geq 1}(Y)\ar[r]&
\Alb^0(Y),
\end{tikzcd}
\end{equation}
where the horizontal arrows are the Albanese morphisms.
By Proposition \ref{2.34}, \eqref{2.36.1} corresponds to \eqref{2.36.2} via \eqref{2.32.1}.
This shows that \eqref{2.36.1} commutes.
\end{proof}

\begin{prop}
\label{2.40}
Suppose $X\in Sm/k$.
The distinguished triangle \eqref{1.15.1}
\[
M_1^*(X)
\to
\uHom_{\DMeffk}(M_{\geq 1}(X),\bZ(1)[2])
\to
\NS(X)
\to
M_1^*(X)[1]
\]
is functorial in $X$.
\end{prop}
\begin{proof}
Suppose $f\colon X\to Y$ is a morphism in $Sm/k$.
We begin with a commutative diagram
\[
\begin{tikzcd}
\uHom_{\DMeffk}(M_{\geq 1}(Y),\bZ(1)[2])\ar[d,"v"']\ar[r]&
\NS(Y)\ar[d,"f^*"]
\\
\uHom_{\DMeffk}(M_{\geq 1}(X),\bZ(1)[2])\ar[r]&
\NS(X),
\end{tikzcd}
\]
where $f^*\colon \NS(Y)\to \NS(X)$ is the naturally induced homomorphism, and $v:=\uHom_{\DMeffk}(M_{\geq 1}(f),\bZ(1)[2])$.
This can be extended to a morphism of distinguished triangles
\[
\begin{tikzcd}
M_1(Y)^*\ar[d,"u"']\ar[r,"a"]&
\uHom_{\DMeffk}(M_{\geq 1}(Y),\bZ(1)[2])\ar[d,"v"]\ar[r]&
\NS(Y)\ar[d,"f^*"]\ar[r]&
M_1(Y)^*[1]\ar[d,"{u[1]}"]
\\
M_1(X)^*\ar[r,"b"]&
\uHom_{\DMeffk}(M_{\geq 1}(X),\bZ(1)[2])\ar[r]&
\NS(X)\ar[r]&
M_1(X)[1]
\end{tikzcd}
\]
for some morphism $u$, where $a$ and $b$ are given by \eqref{1.15.1}.
Consider the morphism $f^*\colon M_1^*(Y)\to M_1^*(X)$.
Proposition \ref{2.36} gives $bf^*=va$.
Since $bu=va$ we deduce
\begin{equation}
\label{2.40.1}
bf^*=bu.
\end{equation}
Due to \eqref{2.29.1} and Theorem \ref{2.30} we have an isomorphism
\[
{\rm Hom}(M_1(Y)^*,M_1(X)^*)
\simeq
{\rm Hom}(M_1(Y)^*,\uHom_\DMeffk(M_{\geq 1}(X),\bZ(1)[2])).
\]
Together with \eqref{2.40.1}, we deduce $f^*=u$.
\end{proof}

\begin{df}
Suppose $X\in Sm/k$.
The \emph{derived Picard of $X$} is defined to be
\[
\RPic(X):=\uHom_{\DMeffk}(M(X),\bZ(1)[2]).
\]
This is a primitive version of the derived Picard of $X$ defined by Barbieri-Viale and Kahn in \cite[Definition 5.3.1]{MR3545132}.
\end{df}

\begin{df}
As noted in Introduction, for every $X\in Sm/k$ the \emph{derived Albanese of $X$} is defined to be
\[
\LAlb(X)
:=
\tau_{\geq 0} \uHom_\DMeffk(\uHom_\DMeffk(M(X),\bZ(1)[2]),\bZ(1)[2]).
\]
Recall that we defined $\NS^*(X):=\uHom_{\DMeffk}(\NS(X),\bZ(1)[1])$.
If $U\to X$ is an open immersion in $Sm/k$, we define
\[
\LAlb(X/U)
:=
\uHom_\DMeffk(\uHom_\DMeffk(M(X/U),\bZ(1)[2]),\bZ(1)[2]).
\]
\end{df}

\begin{thm}
\label{2.41}
Suppose $X\in Sm/k$.
Then in $\DMeffk$, there exists a functorial distinguished triangle
\begin{equation}
\NS^*(X)[1]
\to
\LAlb(X)
\to
\Alb(X)
\to
\NS^*(X)[2].
\end{equation}
\end{thm}
\begin{proof}
Apply $\uHom_{\DMeffk}(-,\bZ(1)[2])$ to \eqref{1.15.1} to have a distinguished triangle
\[
\begin{split}
\NS^*(X)[1]
&\to
\underline{\rm Hom}_\DMeffk(\underline{\rm Hom}_\DMeffk(M_{\geq 1}(X),\bZ(1)[2]),\bZ(1)[2])
\\
&\to
\underline{\rm Hom}_\DMeffk(M_1^*(X),\bZ(1)[2])
\to
\NS^*(X)[2].
\end{split}
\]
This induces a distinguished triangle
\begin{equation}
\label{2.41.1}
\begin{split}
\NS^*(X)[1]
&\to
\underline{\rm Hom}_\DMeffk(\underline{\rm Hom}_\DMeffk(M(X),\bZ(1)[2]),\bZ(1)[2])
\\
&\to
\underline{\rm Hom}_\DMeffk(M_1^*(X),\bZ(1)[2])\oplus M_0(X)
\to
\NS^*(X)[2].
\end{split}
\end{equation}

There is an exact sequence $0\to \bZ^r\to \bZ^s\to \NS(X)\to 0$ for some integers $r,s\geq 0$, and this induces a distinguished triangle
\[
\NS^*(X)[1]\to \bZ^s(1)[2]\to \bZ^r(1)[2]\to \NS^*(X)[2].
\]
This implies
\begin{equation}
\label{2.42.3}
h_{i}(\NS^*(X)[1])=0
\end{equation}
for $i<0$.
Hence we can use the $0$-motivic $t$-structure to \eqref{2.41.1} to have a distinguished triangle
\begin{align*}
\NS^*(X)[1] &\to \LAlb(X)
\\
&\to \tau_{\geq 0}\uHom_\DMeffk(M_1^*(X),\bZ(1)[2])\oplus M_0(X)\to \NS^*(X)[2],
\end{align*}
which can be written as
\begin{equation}
\label{2.42.2}
\NS^*(X)[1]
\to
\LAlb(X)
\to
\Alb(X)
\to
\NS^*(X)[2]
\end{equation}
by Lemma \ref{2.28}, Proposition \ref{2.25}, and Theorem \ref{2.30}.
The functoriality follows from Propositions \ref{2.26} and \ref{2.40}.
\end{proof}

\begin{prop}
\label{2.50}
Suppose $j\colon U\to X$ be an open immersion in $Sm/k$.
Then in $\DMeffk$, there exists a canonical distinguished triangle
\begin{equation}
\label{2.50.1}
\LAlb(U)\to \LAlb(X)\to \LAlb(X/U)\to \LAlb(U)[1].
\end{equation}
\end{prop}
\begin{proof}
We only need to consider the case when $X$ is integral.
In $\DMeffk$, there is a distinguished triangle
\begin{align*}
\LAlb(X/U)
&\to 
\uHom_\DMeffk(\uHom_\DMeffk(M(U),\bZ(1)[2]),\bZ(1)[2])
\\
\to &
\uHom_\DMeffk(\uHom_\DMeffk(M(X),\bZ(1)[2]),\bZ(1)[2])
\to
\LAlb(X/U)[1].
\end{align*}
We can use the $0$-motivic $t$-structure to this distinguished triangle to obtain \eqref{2.50.1} if we have the vanishing
\[
h_0(\LAlb(X/U))=0.
\]
This follows from Proposition \ref{2.49} since $h_0(\bZ(1)[2])=0$.
\end{proof}

\begin{thm}
\label{2.44}
Suppose
\[
\begin{tikzcd}
Y'\ar[d,"f'"']\ar[r,"g'"]&Y\ar[d,"f"]
\\
X'\ar[r,"g"]&X
\end{tikzcd}
\]
is a Nisnevich distinguished triangle in $Sm/k$, i.e., $f$ is \'etale, $g$ is an open immersion, and the induced morphism $f^{-1}(X-g(X'))\to X-g(X')$ with the reduced scheme structure on $X-g(X')$ is an isomorphism.
Then in $\DMeffk$, there is a homotopy cartesian square
\[
\begin{tikzcd}
\LAlb(Y')\ar[d]\ar[r]&
\LAlb(Y)\ar[d]
\\
\LAlb(X')\ar[r]&
\LAlb(X).
\end{tikzcd}
\]
\end{thm}
\begin{proof}
The homotopy cartesian square
\[
\begin{tikzcd}
M(Y')\ar[d]\ar[r]&M(Y)\ar[d]
\\
M(X')\ar[r]&M(X)
\end{tikzcd}
\]
induces an isomorphism $M(Y/Y')\to M(X/X')$.
This induces an isomorphism $\LAlb(Y/Y')\to \LAlb(X/X')$.
Together with Proposition \ref{2.50}, we finish the proof.
\end{proof}

\bibliography{../bib}
    \bibliographystyle{siam}
\end{document}